\documentclass[12pt]{article}
\usepackage[top=1in, bottom=1in, left=1in, right=1in]{geometry}
\usepackage{amsmath}
\usepackage{amsfonts}
\usepackage{amssymb}
\usepackage{amsthm}
\usepackage[usenames, dvipsnames]{xcolor}
\usepackage{tikz}
\usetikzlibrary{automata,topaths,arrows,shapes}
\usepackage{graphicx}
\usepackage{ulem}
\usepackage{parskip}
\usepackage{caption}
\usepackage{pgfplots}
\usepackage{float}
\usepackage{mathrsfs}

\usepackage{wrapfig}

\newtheorem{thm}{Theorem}[section]
\newtheorem{lem}[thm]{Lemma}

\newtheorem{cor}[thm]{Corollary}

\theoremstyle{definition}
\newtheorem{defn}[thm]{Definition}
\newtheorem{rem}[thm]{Remark}
\newcommand{\setof}[1]{\left\{ {#1}\right\}}

\newcommand{\R}{{\mathbb{R}}}

\newcommand{\cA}{{\mathcal A}}
\newcommand{\cB}{{\mathcal B}}
\newcommand{\cC}{{\mathcal C}}

\newcommand{\cE}{{\mathcal E}}
\newcommand{\cF}{{\mathcal F}}

\newcommand{\cK}{{\mathcal K}}

\newcommand{\cM}{{\mathcal M}}

\newcommand{\cV}{{\mathcal V}}
\newcommand{\cW}{{\mathcal W}}

\newcommand{\cZ}{{\mathcal Z}}

\newcommand{\sP}{{\mathsf P}}

\newcommand{\sMG}{\mathsf{MG}}

\newcommand{\sMD}{{\mathsf{ MD}}}

\newcommand{\Int}{\mathop{\mathrm{int}}\nolimits}

\definecolor{gray85}{gray}{0.85} 
\definecolor{gray8}{gray}{0.8} 
\definecolor{gray7}{gray}{0.7} 
\definecolor{gray6}{gray}{0.6} 
\definecolor{gray5}{gray}{0.5} 
\definecolor{gray4}{gray}{0.4} 
\definecolor{gray35}{gray}{0.35} 

\theoremstyle{definition}
\newcommand{\bbR}{\mathbb{R}} 

\newcommand{\sFP}{\mathsf{FP}}

\newcommand{\scrL}{\mathscr{L}}

\newcommand\sgn{\mathtt{Sign}}

\newcommand*{\bfrac}[2]{\genfrac{\{ }{ \}}{0pt}{}{#1}{#2}}

\makeatletter
\tikzset{circle split part fill/.style  args={#1,#2}{%
 alias=tmp@name, 
  postaction={%
    insert path={
     \pgfextra{%
     \pgfpointdiff{\pgfpointanchor{\pgf@node@name}{center}}%
                  {\pgfpointanchor{\pgf@node@name}{east}}%
     \pgfmathsetmacro\insiderad{\pgf@x}
      \fill[#1] (\pgf@node@name.base) ([xshift=-\pgflinewidth]\pgf@node@name.east) arc
                          (0:180:\insiderad-\pgflinewidth)--cycle;
      \fill[#2] (\pgf@node@name.base) ([xshift=\pgflinewidth]\pgf@node@name.west)  arc
                           (180:360:\insiderad-\pgflinewidth)--cycle;            
         }}}}}  
\makeatother

\title{Comparison of  two combinatorial models of global network dynamics}
\author{Peter Crawford-Kahrl, Bree Cummins, Tomas Gedeon}

\begin{document}

\maketitle

\begin{abstract}
Modeling the dynamics of biological networks introduces many challenges, among them the lack of first principle models, the size of the networks, and difficulties with parameterization.
Discrete time Boolean networks and related  continuous time switching systems provide a computationally accessible way to translate the structure of the network to predictions about the dynamics. 
Recent work has shown that the parameterized dynamics of switching systems can be captured by a combinatorial object, called a DSGRN database, that consists of a parameter graph characterizing a finite parameter space decomposition, whose  nodes are  assigned  a Morse graph that captures global dynamics for all corresponding parameters. 

We show that for a given network there is a way to associate the same type of object by considering a continuous time ODE system with a continuous right-hand side, which we call an L-system. The main goal of this paper is to compare the two DSGRN databases for the same network. Since the L-systems can be thought of as perturbations (not necessarily small) of the switching systems, our results address the correspondence between global parameterized dynamics of switching systems and their perturbations. 
We show that, at corresponding parameters, there is an order preserving map from the Morse graph of the switching system to that of the L-system that is surjective on the set of attractors and bijective on the set of fixed point attractors. We provide important examples showing why this correspondence cannot be strengthened. 
\end{abstract}

\section{Introduction}
Nonlinear dynamics is notoriously difficult. Since a complete rigorous analysis  of dynamics of nonlinear systems of coupled  differential equations of dimension higher than two is almost impossible, any interrogation of higher dimensional systems of ODE usually  relies on  numerical simulations.

In problems arising in cellular biology, there is a need to model many mutually interacting types of molecules that together control cellular fate.  Incorrectly functioning  genetic and regulatory networks are at the core of cancer, diabetes, and other systemic diseases~\cite{Burkhart:08,Chinnam:11,Manning:12}, The crucial importance of these networks for cell biology and human health makes development of effective methods that can characterize dynamics  supported by a network over all parameters a high priority~\cite{albert:collins:glass,Randrup:04,Ingram:06,Alon,Shamir2008,Ma09,Shah11}.
In the context of cell biology, the problems of nonlinear dynamics are compounded by the lack of first principles that would determine appropriate nonlinearities,
the difficulty in obtaining precise experimental data needed to determine parameters, and the need to analyze the dynamics of  5-10 dimensional systems over 30-50 parameters.  

Recently, we introduced in~\cite{us2,Cummins17} a new approach to this problem that assigns two finite objects to any network with  positive and negative edges.  First is a {\it parameter graph} whose nodes are in 1-1 correspondence with regions in the parameter space, where these regions  form a decomposition of the parameter space.  To each domain of the parameter space, i.e. a node of the parameter graph, there is an associated {\it state transition graph} that characterizes allowable transitions between well-defined states of the phase space. Since state transition graphs can be large, a useful description of the recurrent trajectories  is a \textit{Morse graph}, which is graph of strongly connected path components of the state transition graph. 
The entire structure, where to each node of the parameter graph there is an associated Morse graph that captures recurrent dynamics valid for all parameters in the corresponding parameter region, is called Database of Signatures Generated by Regulatory Networks (DSGRN). 

The advantages of such a description of global dynamics is its finiteness and the resulting computability; yet this description, which inevitably must be coarser than the traditional concepts of dynamical systems theory, allows  searching  the database for parameters that support dynamics like bistability, hysteresis, non-constant recurrent behavior, and the ability to compare the prevalence of such signatures across multiple networks.

The development of DSGRN was guided by work over the last two decades on {\it switching systems} \cite{glass:kaufman:73,Thomas1991,Thomas95,Gouze2002,deJong2004,edwards00} which are ordinary differential equations with piecewise constant nonlinearities. The value of each  nonlinearity changes discontinuously when an argument crosses  a threshold. The collection of these thresholds divides the phase space into domains, which form the nodes of the state transition graph. The choice of piecewise constant nonlinearities presents two sets of challenges. The mathematical challenge is to make sense of the continuation of solutions that enter the intersection of multiple thresholds. The biological challenge is to justify the selection of piecewise constant nonlinearities as appropriate models of biological processes.
Our view is that DSGRN, being by its construction a finite, computable and robust object, gives us information not only about the switching system that was used to construct it, but it also describes all nearby continuous systems~\cite{us1}. 

There has been a considerable interest in the computational biological systems community in trying to enlarge a class of ODE systems for which finite state transition graphs capture the behavior of all solutions~\cite{Batt2005,Belta2006,Batt2007a,Batt2007b}. The result of these papers show how to construct a phase transition graph for so called multi-affine systems, where nonlinearities are piecewise linear functions. 

In this paper we generalize their result to nonlinearities that are step functions with Lipschitz continuous bridges, which we call L-functions. These functions have alternating intervals where the function is constant, and the intervals where function is Lipschitz and bounded between the (constant) values of the function on neighboring intervals. We show that the dynamics of such systems are captured by a state transition graph, and 
consequently one can associate DSGRN results to a network based on interactions mediated by L-functions.
The central question that we address in this paper is relationship between global dynamics of two sets of models, as captured by DSGRN description.
Given a regulatory network, we can associate to it two different  DSGRN databases characterizing global dynamics across all parameters: one based on 
a switching system description, and the other based on  L-functions. 
A natural question is  how these two objects are related.  This is a global version of a question that has been investigated before where a piecewise constant nonlinearity is 
perturbed around a threshold to form a continuous function~\cite{deJong2004,Ironi2011,us1}.


  We begin by introducing a general framework for the construction of state transition graphs. This framework encompasses the well-known case of Boolean maps~\cite{Thomas1973,Thomas1991,Thomas95,Chaves2006,Steinway:2015vv}, and extends it to a \textit{multi-level discrete map} $D$, which  increases the number of discrete states available to each node in the network~\cite{Thieffry06,Chaouiya06}. We then introduce the idea of a \textit{nearest neighbor multi-valued map} $\cF$, which may arise as an \textit{asynchronous update rule} of $D$ in a way analogous to that  discussed in~\cite{Chaves2006} for a Boolean model. $\cF$ obeys an adjacency condition which allows only one node of the network  to change its  state at a time. 

We then formally introduce S-systems (switching systems) and L-systems (ODE systems based on L-functions), and demonstrate the relationship between maps $D$, $\cF$, and these ODE systems. 
We show that there is a map between the parameters of the S- and L-systems, and that under this map, there is a well-defined relationship between the Morse graphs of the two systems. In particular, the map from the S-system Morse graph to the L-system Morse graph is bijective on fixed points, surjective on attractors, and order-preserving otherwise. We conclude with a series of examples that demonstrate that these  relationships cannot be strengthened.

\section{General System}\label{sec:general}

\begin{defn}
A \textit{regulatory network} $\textnormal{\textbf{RN}} = (V,E)$ is a graph with network nodes $V = \{1,2,\dots,N\}$ and signed, directed edges  $E \subset V \times V \times \{\rightarrow,\dashv\}$. For $i,j \in V$, we will use the notation $(i,j) \in E$ to denote a directed edge from $i$ to $j$ of either sign, $i \to j$ to denote an \textit{activation} or positive interaction, and $i \dashv j$ to denote a \textit{repression} or negative interaction. 
We only consider regulatory networks with  no negative self-regulation, $i \dashv i$.

We define the \textit{targets} of a node $i$ as
\begin{equation*}
\mathbf{T}(i):=\{j \mid (i,j) \in E \}
\end{equation*}
and the  \textit{sources} of a node $i$ as
\begin{equation*}
\mathbf{S}(i):=\{j \mid (j,i) \in E \}
\end{equation*}
\end{defn}

We exclude networks with negative self-regulation because they present technical difficulties in switching systems; see~\cite{Edwards2015} for an excellent set of references. In many cases,  negative self-regulation can be replaced in a model by an intermediate node~\cite{Edwards2015,huttinga2017global} which obviates the need for a negative self-edge in a regulatory network.

As an example regulatory network \textbf{RN} we consider $\{x \dashv y, y \dashv x\}$. Here each node has as a target the other node.

\begin{figure}[h]
\centering
	\begin{tikzpicture}[main node/.style={circle,fill=white!20,draw,font=\sffamily\normalsize\bfseries}, scale=1]
		\node[main node] (x) at (0,0) {$x$};
		\node[main node] (y) at (1.5,0) {$y$};	
		\path[->,>=angle 90,thick]
		(x) edge[shorten >= 3pt,shorten <= 3pt,-|,bend right] node[] {} (y)
		(y) edge[shorten >= 3pt,shorten <= 3pt,-|,bend right] node[] {} (x)
		;
	\end{tikzpicture}\caption{Two dimensional example \textbf{RN}.}\label{fig:exampleRN}
\end{figure}
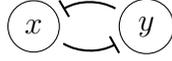

One way to associate dynamics to a  network is to construct a Boolean net~\cite{Thomas1973,Thomas1991,Thomas95,Chaves2006,Steinway:2015vv}. Each node can attain values $0$ or $1$ that are interpreted as low and high levels of activity.  At each node $i$ with $|\mathbf{S}(i)| = n$, there is an associated local Boolean function that assigns to each of the $2^n$ input binary sequences a value of $x_i \in \{0,1\}$.
The collection of the local Boolean functions over the network forms a Boolean function $B: \{0,1\}^N \to \{0,1\}^N$, that acts on a space  of  binary sequences of length $N$. Iterations $B^r$  of this function model long-term behavior of the network. 
The collection of all local Boolean functions that can be selected at each node parameterize the set of all Boolean functions $\cB_\textbf{RN}= \{B\}$  compatible with the given network \textbf{RN}.  Since both the domain $  \{0,1\}^N$ of $B$ and this parameterization of $\cB_\textbf{RN}$ are discrete sets, there is no concept of a ``small'' perturbation from one Boolean function to another.

Motivated by this example, we now propose a different way  to associate to a  network a dynamical system on a finite state  space. These dynamical systems will be parameterized by a continuous parameter space, and so it makes sense to ask how these finite dynamical systems behave under perturbations.
Even though parameter space is continuous, we will show that it  can be divided into a finite number of regions where the behavior of the dynamics is the same, enabling a global description of the network over all real-valued parameters.

We start by  assuming that to each node  of the network there is an associated variable $x_i \in [0, \infty)$, which represents the concentration of chemical species $i$. 
We assume that  there are finite number of thresholds $\theta_{1,i}, \ldots, \theta_{m_i,i}$ 
that divide the semi-axis $[0, \infty)$ to $m_i+1$ intervals $I_k$. The effect of node $i$ on its target nodes $j \in \mathbf{T}(i)$ will only depend on the interval $I_k$ with $x_i \in I_k$ and not on the particular value $x_i$. The collection of thresholds $\{\theta_{j,i}\}$ partitions  $[0,\infty)^N$ into a finite number of domains $\kappa$, characterized by the property that  the projection on $i$-th variable $\pi_i(\kappa) = I_{k}$ for a unique $k \in \{0,\dots,m_i\}$ for every $i$. We let $x=(x_1,\dots,x_N)$ denote a point in $[0,\infty)^N$.

\vspace{12pt}

\begin{defn} \label{defn:GFMD}
Let $\cV(i):=  \{0,\dots,m_i\}$ and let 
\begin{equation}\label{eq:G}
 G_i: [0,\infty) \to \cV(i)  
 \end{equation}  
be a \textit{state function} defined by  $G_i(x_i) = k$ if and only if $x_i \in I_k$.  Let $\cV=\prod_i \cV(i)$ be the set of all \textit{states} of the network {\textbf{RN}} and let
\[G: [0,\infty)^N \to \cV\] be the vector-valued function with coordinate functions $G_i$. 
 For a given domain $\kappa$, the value $G(x)$ does not depend on $x \in \kappa$. Therefore we  can assign  the  \textit{state} $s:=G(x) \in \cV, x \in \kappa$ to the domain $\kappa$, and we will write $s= g(\kappa)$. Viewed as a map on  the set of domains $\cK = \{\kappa\}$  in $[0,\infty)^N$,  $g$ is a bijection between domains $\kappa$ and states $s \in \cV$
 \begin{equation}\label{eq:multilevel_g}
 g : \cK \mathop{\longrightarrow} \cV.
 \end{equation}

We postulate that along each edge $(i, j)$ in the network a signal from the node $i$ to the node $j$ can be transmited, and this  signal attains only  finite number of values $A_{j,i} := \{a_{j,i}^1< \ldots <a_{j,i}^t\}$.  This transmission is characterized  via a function 
\begin{equation}\label{eq:F} 
 F_{j,i} : \cV(i) \to A_{j,i} 
\end{equation}
that only depends on the state $k \in \cV(i)$ of  $x_i$.
 Let $A= \prod A_{j,i}$ be a product of all sets $A_{j,i}$ and let 
\[ F: \cV \to A
\]
be the vector-valued function with coordinate functions $F_{j,i}$.

A final piece of our description of discrete dynamics on a network {\textbf{RN}} is  a collection of functions $M_i$, one for each node $i$ of the network, 
\begin{equation}\label{eq:M} 
 M_i: \prod\limits_{j=1}^{|\mathbf{S}(i)|}A_{i,j}\to [0,\infty) \end{equation}
that take the values that are being transmitted along the edges leading into $i$ and produce the values $x_i$. 
Let $M$ be the vector-valued function with coordinate functions $M_i$, 
\[ M : A \to [0,\infty)^N. 
\]
 The composition 
\begin{equation}\label{eq:D} 
D:= G\circ M \circ F: \cV \to \cV 
\end{equation}
is a generalization of the Boolean function $B$, called a \textit{multi-level discrete function}~\cite{Thieffry06,Chaouiya06}.

\end{defn} 

The set of values $A$ in Definition~\ref{defn:GFMD} replaces the binary values $0,1$ that are  transmitted in a Boolean network.
Note that even when two outward edges from node $i$ have the same sign, say $i \to j$ and $i \to l$, different values can be transmitted to $j$ and $l$  since the values $F_{j,i}(k) $ and $ F_{l,i}(k)$ may be different.  This  generalizes the behavior of a traditional Boolean function.


One serious objection to representing the dynamics of a network by either a Boolean function $B$ or a discrete function $D$ is that it does not respect the continuity of the underlying biological process. In particular, the Boolean vector $(B(s) - s)$ can be non-zero in more than one component which implies that two or more processes switch at exactly the same time.  In addition,  the vector $(D(s)-s)$ can have entries greater than 1 in absolute value, 
which clearly violates continuity of the underlying chemical process. 
An  {\it asynchronous update} of Boolean function $B$ has been proposed~\cite{Chaves2006,Chaves09} to generate dynamics that are compatible continuous variables, and we extend this approach to map $D$.


For a given multi-level discrete map $D$ we  define a {\it nearest neighbor multi-valued map} $\cF$, that  will only allow transitions from domains to the adjacent domains in phase space, where these  transitions are induced by $D$.


\vspace{12pt}

\begin{defn}\label{def:asynchronousupdate}
{\rm
	Let $s_1$ and $s_2$ be the states of two domains $\kappa_1$ and $\kappa_2$, $s_1 = g(\kappa_1)$ and $s_2=g(\kappa_2)$ from~\eqref{eq:multilevel_g}. These domains are adjacent along $i$ (and so are the states) if and only if there exists exactly one index $i$ such that  
	\[ \pi_i(\kappa_1) \cap \pi_i(\kappa_2) \subset \{ x_i = \theta_{j,i} \} \mbox{ and } \pi_j(\kappa_1) = \pi_j(\kappa_2) \mbox{ for all } j \not = i.\]
  Let  $t_1 = D(s_1)$.
	The \textit{asynchronous update rule of $D$} is a nearest neighbor multi-valued map $\cF : \cV \rightrightarrows \cV$, such that 
	 $s_2 \in \cF(s_1)$ if and only if 
	 \begin{description}
	 \item[(a)] $s_1 = t_1=s_2$; or 
	 \item[(b)] $t_1 \not = s_1$, $s_1 :=(s_{1,1},s_{1,2},\dots,s_{1,N})$ and $s_2 :=(s_{2,1},s_{2,2},\dots,s_{2,N})$ are adjacent along $i$ 
	 and
	 either
	\begin{enumerate}
	\item $s_{1,i} < s_{2,i} \leq t_{1,i}$ or
	\item $s_{1,i} > s_{2,i} \geq t_{1,i}$.
	\end{enumerate}
	
	\end{description}
  $\cF$ is sometimes represented as a graph, called a \textit{state transition graph},  $(\cV,\cE)$, where $(s,t) \in \cE$ if and only if $t \in \cF(s)$.
  }
		\end{defn}

We note that nearest neighbor multi-valued maps $\cF$ can be constructed without reference to an underlying multi-level discrete map $D$ by considering an arbitrary multi-valued map 
$\cF : \cV \rightrightarrows \cV$ where the image $\cF(s)$ is adjacent to $s$.

%


Since the number of states in $\cV$ can be  very large,  the dynamics of iterates of $\cF$ can be captured by  a more  compact representation~\cite{Chaves09,us2}. 

\vspace{12pt}

\begin{defn}\label{def:recurrent} 
A \textit{recurrent component}   of the map  $\cF$  is a \textit{strongly connected path component} of the associated   graph $(\cV,\cE)$. In other words, it is a maximal collection of vertices $\cC \subset \cV$ such that for any $u, v \in \cC$ there exists a non-empty path from $u$ to $v$ with vertices in  $\cC$ and edges in $\cE$.   In the context of dynamical systems we refer to a recurrent component of $\cF$ as a \textit{Morse set} of $\cF$ and denote it by $\cM\subset \cV$. The collection of all recurrent components of $\cF$ is denoted by 
\[
\sMD(\cF) :=\setof{\cM(p)\subset \cV\mid p\in \sP}
\]
and is called a \textit{Morse decomposition} of $\cF$, where $\sP$ is an index set. Recurrent components inherit a well-defined partial order by the reachability relation in the directed graph $(\cV,\cE)$. Specifically, we may write the partial order on the indexing set $\sP$ of $\sMD(\cF)$  by defining
\[
q \leq p\quad \text{if there exists a path in $(\cV,\cE)$ from an element of $\cM (p)$ to an element of $\cM (q)$}.
\]
\end{defn}
%

\vspace{12pt}

\begin{defn} \label{defn:morsegraph}   
The \textit{Morse graph} of $\cF$, denoted $\sMG(\cF)$, is the Hasse diagram of the poset $(\sP,\leq)$. We refer to the elements of $\sP$ as the \textit{Morse nodes} of the graph.
\end{defn}

An intriguing question is the characterization of the set of ordinary differential equations models that are compatible with a given map  $\cF$. 
This class of equation will share the same broad dynamical features that are captured by the  Morse graph of $\cF$. In the other direction, identifying a map $\cF$ for a given ODE system would facilitate its analysis, because of the inherent computability of the Morse graph from the map $\cF$.

\vspace{12pt}

\begin{defn}\label{defn:compatible}
We say an ordinary differential equation model with variables $x_i$ is \textit{compatible} with a nearest neighbor multi-valued map $\cF$  if solutions $x(t)$ can traverse  from domain $\kappa_1$ to adjacent domain $\kappa_2$ only if $s_2 \in \cF(s_1)$. 
\end{defn}

	In this manuscript, we offer two examples of  ordinary differential equation models associated to  a regulatory network that are compatible with  a nearest neighbor map  $\cF$, which we call S-systems and L-systems.
	The S-system, also known as a switching system in the literature, has been very well studied, while the ability to define $\cF$ associated to  the L-system is a new contribution. In both cases, the association between the continuous ODE system and a discrete map $\cF$ allows us to combine the best features of both worlds. On one hand, 
	there is a combinatorial representation of both the dynamics and the parameters, which allows computational enumeration of all types of dynamics for all parameters.
	On the other hand,  the underlying assumptions of continuity  allow us to interpret the dynamics of iterates of $\cF$ in terms of  solutions of  systems of ordinary differential equations.

\subsection{S-system} \label{defn:params_S}

Given a regulatory network $\mathbf{RN} = (V,E)$, 
to  each node $i$ we assign a parameter $\gamma_i^S$, which will be interpreted as a rate of degradation of $x_i$.
For each edge $(i,j) \in E$  we associate three numbers: a threshold $\theta_{j,i}$ and a low value $l_{j,i}^S$ and a high value $u_{j,i}^S$, so that 
$A_{j,i} = \{ l_{j,i}^S, u_{j,i}^S\}$ (see Definition~\ref{defn:GFMD}). We require for all $i$ that
\begin{align*} 
& \quad 0 < \gamma_i^S, \quad 0 < l_{j,i}^S < u_{j,i}^S, \quad 0 < \theta_{j,i}, \quad \theta_{j,i} \neq \theta_{k,i} \mbox{ whenever } j \neq k
\end{align*}
and we call an \textit{S-parameter}  of $\mathbf{RN}$ the tuple $z^S=(l^S,u^S,\theta,\gamma^S) \in \bbR^{d^S}$ where $d^S = \#(V) + 3\#(E)$. 
The collection of the threshold parameters $\{\theta_{j,i}\}$  partitions $[0,\infty)^N$, which facilitates the definition of the function $G^S : [0,\infty)^N \to  \cV^S$, 
$\cV^S = \prod_i \cV(i)$,  via its component functions 
\begin{align*} 
G_i^S : [0, \infty)\setminus\{\theta_{j,i}\} \to \cV(i) = \{0,\dots,|\textnormal{\textbf{T}}(i)|\}
\end{align*} 
as in Definition~\ref{defn:GFMD}.
Let $\theta_{j,i}(k)$ be $k$-th threshold in the linearly ordered set $\{ \theta_{\ell,i} : \ell \in \textnormal{\textbf{T}}(i)\} \subset [0,\infty)$. 
Then the function $G_i^S$ is defined by
\[ x_i < \theta_{j,i} \Leftrightarrow G_i^S(x_i) \leq k-1; \quad \quad x_i >  \theta_{j,i} \Leftrightarrow G_i^S(x_i) > k-1.\]

We define  $F_{j,i}^S$ from~\eqref{eq:F} as
\begin{align}
 F_{j,i}^S \,\circ\, G_i^S(x_i)  :=&
\begin{cases}
l_{j,i}^S & \text{if  } G_i^S(x_i) \leq k-1 \text{ and } i\to j,\text{ or } G_i^S(x_i) > k-1 \text{ and } i\dashv j\\
u_{j,i}^S & \text{for } G_i^S(x_i) > k-1 \text{ and } i\to j,\text{ or } G_i^S(x_i) \leq k-1 \text{ and } i\dashv j\\
\text{undefined} & \text{otherwise}
\end{cases}\label{eq:sigmaswitching}
\end{align}
Let $\sigma_{j,i}^S := F_{j,i}^S \,\circ\, G_i^S(x_i)$ and 
where $\sigma_j^S$ is the vector-valued function with components $\sigma_{j,i}^S$.


In addition, to every node $j$  we assign a rule $\bar M_j$ that combines all the values of the input edges to a node $j$  into a single real value $x_j$. This rule is called a \textit{logic} at the node $j$ and it is assumed to be a multi-affine function with all coefficients equal to $1$. Recall that a multi-affine function is  a polynomial with the property that the degree in any of its variables is at most 1. 

The  \textit{S-system for \textbf{RN}} is a system of ordinary differential equations 
\begin{equation}	\label{generalswitchingsystemequation}
\dot{x}_j=-\gamma_j^S x_j + \Lambda_j^S(x) = -\gamma_j^S x_j + \bar M_j \circ \sigma_{j}^S(x) \qquad	j=1,\dots,N
\end{equation}

We will later show the connection between this system and a multi-level discrete function $D^S$ as in~\eqref{eq:D}, and its  asynchronous update rule $\cF^S$ as in Definition~\ref{def:asynchronousupdate}.

Continuing the example \textbf{RN} shown in Figure~\ref{fig:exampleRN}, the corresponding S-system is the system given by
\begin{align*}
\dot x &= -\gamma_x^S + \sigma_{x,y}^S(y)\\
\dot y &= -\gamma_y^S + \sigma_{y,x}^S(x)
\end{align*}
where 
\begin{equation*}
\sigma_{x,y}^S(y) = \begin{cases} u_{x,y}^S & \text{if } y < \theta_{x,y} \\
l_{x,y}^S & \text{if } y > \theta_{x,y}^S
\end{cases}; \qquad \sigma_{y,x}^S(x) = \begin{cases} u_{y,x}^S & \text{if } x < \theta_{y,x} \\
l_{y,x}^S & \text{if } x > \theta_{y,x}
\end{cases}.
\end{equation*}
The function  $\sigma_{y,x}^S(x)$ is depicted on the left of Figure~\ref{fig:exampleSandLsystem}. The other function  $\sigma_{x,y}^S(y)$ will have the same shape, as both edges of the example \textbf{RN} correspond to negative regulation.

\subsection{L-system} \label{defn:params_L}

For the L-system we replace a single  threshold $\theta_{j,i}$ by two thresholds  $\vartheta_{j,i}^-$, and $\vartheta_{j,i}^+$. 
Given a regulatory network $\mathbf{RN} = (V,E)$, to  each node $i$ we again assign a decay parameter $\gamma_i^L$.  For each edge $(i,j) \in E$, we associate four real-valued parameters  $u_{j,i}^L$, $l_{j,i}^L$, $\vartheta_{j,i}^-$, and $\vartheta_{j,i}^+$. Here again $A_{j,i} := \{l_{j,i}^L, u_{j,i}^L\}$  (Definition~\ref{defn:GFMD}).
We require for all $i$ that
\begin{align*} 
& \quad 0 < \gamma_i^L, \quad 0 < l_{j,i}^L < u_{j,i}^L, \quad 0 < \vartheta_{j,i}^- < \vartheta_{j,i}^+, \quad [\vartheta_{j,i}^-,\vartheta_{j,i}^+] \cap [\vartheta_{k,i}^-,\vartheta_{k,i}^+] = \emptyset \mbox{ whenever } j \neq k.
\end{align*}

 The tuple $z^L=(l^L,u^L,\vartheta^-,\vartheta^+,\gamma^L) \in \bbR^{d^L}$ is  an L-parameter of $\mathbf{RN}$, where $d^L = \#(V) + 4\#(E)$.

In an analogy with the S-system, 
we define a  function 
\begin{equation}\label{eq:sigmaL}
\sigma_{j,i}^L(x):= \pi_i(\sigma_j^L(x)) =
\begin{cases}
l_{j,i}^L	&	\text{for } x_i \leq \vartheta_{j,i}^- \text{ and } i\to j, \text{or } x_i\geq\vartheta_{j,i}^+ \text{ and } i\dashv j\\
u_{j,i}^L	&	\text{for}\ x_i \geq \vartheta_{j,i}^+ \text{ and } i\to j, \text{or } x_i\leq\vartheta_{j,i}^- \text{ and } i\dashv j\\
 f_{j,i}^L(x_i)  & \text{for } x_i \in [\vartheta_{j,i}^-,\vartheta_{j,i}^+]
\end{cases}
\end{equation}
where $f_{j,i}^L(x_i)$ is a Lipschitz continuous function with lower bound $l^L_{j,i}$ and upper bound $u^L_{j,i}$. 
The function $\sigma_{j,i}^L(x)$ is a step function that is  regularized by a Lipschitz bridge. This  defines a vector-valued function $\sigma_j^L(x)$ coordinate-wise.

\vspace{12pt}

\begin{defn}\label{def:GforLsys}
At  every L-parameter and for every $i= 1, \ldots, N$, the interval $[0,\infty)$ is decomposed into intervals with non-overlapping interiors
\[ (I_{0,i} := [0,\vartheta^-_{j_1,i}]) \leq (I_{\frac{1}{2},i} := [\vartheta^-_{j_1,i},\vartheta^+_{j_1,i}]) \leq (I_{1,i} := [\vartheta^+_{j_1,i},\vartheta^-_{j_2,i}]) \leq  \cdots \leq (I_{m_i,i} := [\vartheta^-_{j_{m_i},i},\infty)). \]
We define the $i$-th component of the  state function $G^L$ (see~\eqref{eq:G}) by 
\[ G^L_i(x_i) = k \mbox{ when } x_i \in \Int I_{k,i}. \] 

We leave $G^L_i(x_i)$ undefined on finite set of values $x_i= \vartheta^\pm_{n,i}$ where $n = j_1, \ldots, j_{m_i}$.
Let $\cV^L(i) := \{ 0, \frac{1}{2}, 1, \frac{3}{2}, \dots, m_i \}$ be the range of $G^L_i$ and let $\cV^L:= \prod_i \cV^L(i)$ be the set of states associated to the domains $\cK^L$.
Note that $G^L$ induces a bijection $g^L$  between the set of  domains $\cK^L$ and $\cV^L$, as in~\eqref{eq:multilevel_g}.
\end{defn}


A system of ordinary differential equations is called an \textit{L-system associated to  \textbf{RN}} if 
\begin{equation}	\label{generalLsystemequation}
\dot{x}_j=-\gamma_j^L x_j + \Lambda_j^L(x) = -\gamma_j^L x_j + \bar M_j \circ \sigma_{j}^L(x) \qquad	j=1,\dots,N,
\end{equation}
where $\bar M_j$ is defined as for an S-system.

We continue our example. For  network shown in Figure~\ref{fig:exampleRN}, the associated L-system is given by 
\begin{align*}
\dot x &= -\gamma_x^S + \sigma_{x,y}^L(y)\\
\dot y &= -\gamma_y^S + \sigma_{y,x}^L(x).
\end{align*}

 In Figure~\ref{fig:exampleSandLsystem} we depict one possible shape of  function $\sigma_{i,j}^L(x_j)$; any  continuous function $f^L_{i,j}(x_j)$ that connects $u^L_{y,x}$ and $l^L_{y,x}$ and is bounded vertically between these values satisfies our constrains on $\sigma_{i,j}^L(x_j)$. 

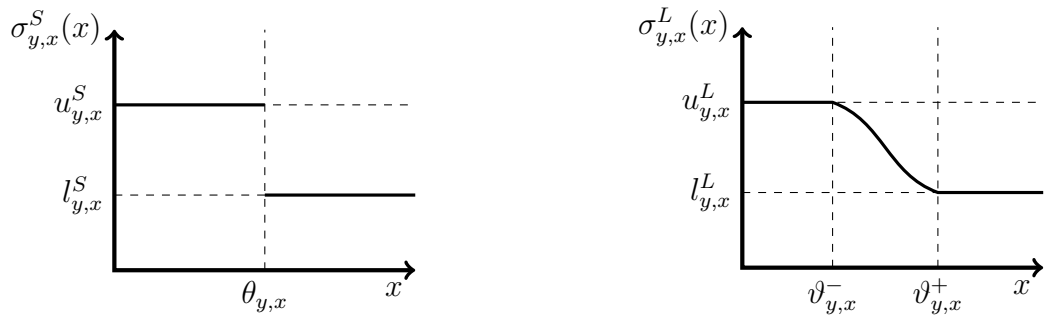
\begin{figure}[h!]
\begin{minipage}{.5\textwidth}
\begin{tikzpicture}[scale=.8]
  \draw[ultra thick,<->] (0,4) node[anchor=east]{$\sigma_{y,x}^S(x)$} -- (0,0) -- (5,0) node[anchor=north east]{$x$} ;
  \draw[very thick] (0,2.75) -- (2.5,2.75);
  \draw[very thick] (2.5,1.25) -- ++(2.5,0);
  \draw[dashed] (2.5,0) node [below]{$\theta_{y,x}$} -- ++(0,4);
  \draw[dashed] (0,1.25) node[anchor=east]{$l_{y,x}^S$} -- ++(5,0);
  \draw[dashed] (0,2.75) node[anchor=east]{$u_{y,x}^S$}-- ++(5,0);
\end{tikzpicture}
\end{minipage}
\begin{minipage}{.5\textwidth}
\begin{tikzpicture}[scale=.8]
  \draw[ultra thick,<->] (0,4) node[anchor=east]{$\sigma_{y,x}^L(x)$} -- (0,0) -- (5,0) node[anchor=north east]{$x$} ;
  \draw [shorten >= 0pt,shorten <= 0pt,very thick] 
  (0,2.75) to (1.5,2.75) to[out=340, in=160] (3.25,1.25) to (5,1.25);  
  \draw[dashed] (1.5,0) node [below]{$\vartheta_{y,x}^-$} -- ++(0,4);
  \draw[dashed] (3.25,0) node [below]{$\vartheta_{y,x}^+$} -- ++(0,4);
  \draw[dashed] (0,1.25) node[anchor=east]{$l_{y,x}^L$} -- ++(5,0);
  \draw[dashed] (0,2.75) node[anchor=east]{$u_{y,x}^L$}-- ++(5,0);
\end{tikzpicture}
\end{minipage}
\caption{(Left) Function $\sigma_{y,x}^S(x)$ from the  S-system associated to \textbf{RN} shown in Figure~\ref{fig:exampleRN}. (Right)  Function  $\sigma_{y,x}^L(x)$ from the associated L-system.} \label{fig:exampleSandLsystem}
\end{figure}

It is important to note that the function $\sigma_{j,i}^L(x)$ cannot be  represented as a composition  $F_{j,i}^L \circ G_i^L$ as in the S-system. This is because the range of $\sigma_{j,i}^L$ is an interval  $[l_{j,i}^L,u_{j,i}^L]$, rather than the discrete set of values $\{l_{j,i}^L,u_{j,i}^L\}$. This means that we cannot construct a multi-level discrete function $D^L$. However, we will construct a nearest neighbor multi-valued function $\cF^L$ such that the solutions of L-system are compatible with $\cF^L$.

Systems of ordinary differential equations similar to the L-system have been studied as continuous perturbations of the S-systems~\cite{Ironi2011,us1}.  In this interpretation, the interval 
$[ \vartheta_{j,i}^-,  \vartheta_{j,i}^+]$ in the definition of $ \sigma_{j,i}^L$   contains the threshold $\theta_{j,i}$ of $ \sigma_{j,i}^S$, and has length $\epsilon$, a small number. With the same values $l_{j,i}^L = 	l_{j,i}^S$, and 	$u_{j,i}^L = u_{j,i}^S$, the function $ \sigma_{j,i}^L$ is a small $C^0$ perturbation of $ \sigma_{j,i}^S$. 
A challenge is to characterize how the dynamics of such a nearby L-system reflect the dynamics of the S-system.  This is a difficult question in the ODE setting, where the emphasis is on individual trajectories.

In this paper we address this question   from a perspective of global dynamics, where we compare the Morse graphs associated to $\cF^L$ and $\cF^S$. Furthermore, we do not require that intervals $[ \vartheta_{j,i}^-,  \vartheta_{j,i}^+]$ are small; this is replaced by the requirement that  $ [\vartheta_{j,i}^-,\vartheta_{j,i}^+] \cap [\vartheta_{k,i}^-,\vartheta_{k,i}^+] = \emptyset$.

\section{Construction of $\cF^S$ and $\cF^L$}\label{sec:domaingraphs}

In this section we will show that both the S- and L-systems generate nearest neighbor multi-valued functions  $\cF^S$ and $\cF^L$. We will  represent these maps as graphs with vertices that correspond to discrete states, and edges that correspond to allowed  transitions.
\vspace{12pt}

It is clear from the definition of the S-system
that  the thresholds $\{\theta_{j,i} : j \in \mathbf{T}(i)\}$ form a strict total order for each $i \in V$. We denote this collection of orderings by $O(z^S)$. Similarly, the 
intervals $\{[\vartheta_{j,i}^-,\vartheta_{j,i}^+] : j \in \mathbf{T}(i)\}$ form a strict total order for each $i \in V$, and we denote the collection by $O(z^L)$.
Note that 
$\Lambda^S_j$ and $\Lambda^L_j$ are multi-affine combinations of  bounded functions, so they are themselves bounded. 

For convenience we introduce the thresholds $\vartheta_{0,i} = \theta_{0,i}=0$ and $\vartheta_{\infty,i} = \theta_{\infty,i} = \infty$ for each $i$. \vspace{12pt}

\begin{defn} \label{def:cell}
Let $\varphi, \varphi'$ be thresholds in either the switching or perturbed systems. We say that $\varphi, \varphi'$ are \textit{adjacent} if $\varphi < \varphi'$ and there does not exist $\varphi''$ such that $\varphi < \varphi'' < \varphi'$.

	Let 
	\[\zeta := \prod^N_{i=1} I_i   \] 
where  $I_i$ is either a non-degenerate interval $I_i=  [\varphi_i ,\varphi'_i]$  with adjacent thresholds $\varphi_i ,\varphi'_i$,  a half-infinite interval $I_i=  [\varphi_i ,\infty)$ where $\varphi_i$ is the largest of the thresholds of $x_i$, or a degenerate interval
	$I_i =[\varphi_i, \varphi_i]$.
Let 
\[ND(\zeta):= \{ i \in \{1,\dots,N\} \mid \text{$I_i$ is a non-degenerate interval}\},\]
 and let $\ell = \#(ND(\zeta))$. Note  that $\zeta$ has dimension $\ell$ in phase space and say that $\zeta$ is an \textit{$\ell$-cell}.
	
If $\zeta$ is an $N$-cell we refer to it as a \textit{domain}.  The collection of all  domains of S-systems  will be denoted by  $\cK^S$ and the same collection  for L-system  will be $\cK^L$.

In order to facilitate comparison between domains of S- and L-systems we define two subsets of $\cK^L$. First, we denote
$\cK_{N}^L \subsetneq \cK^L$ to be the set of domains $\kappa$ such that for every $i\in ND(\kappa)$, the interval $I_i$ is either half-infinite or of the form $I_i = [ \vartheta_{j_{m},i}^+, \vartheta_{j_{m+1},i}^-]$. These are the intervals where the functions $\sigma_{j,i}^L(x_i)$ are constant. Second, we define the subset $\cK_{N-1}^L \subsetneq \cK^L$ such that one and exactly one $\sigma_{j,i}^L(x_i) =f_{ji}^L(x_i)$ is not constant; that is, there is an exactly one  $i \in ND(\kappa)$ such that $I_i = [ \vartheta_{j,i}^-, \vartheta_{j,i}^+]$ for some $j$.
\end{defn}

\vspace{12pt}

\begin{figure}[h!]
\centering
\begin{tikzpicture}[scale=.75]
	\draw[ultra thick,<->] (0,4) node[anchor=north east]{$y$} -- (0,0) -- (5,0) node[anchor=north east]{$x$} ;
	\draw[dashed] (2.5,0) node[anchor=north]{$\theta_{y,x}$} -- (2.5,4);
	\draw[dashed] (0,2) node[anchor=east]{$\theta_{x,y}$} -- (5,2);

\draw[ultra thick,<->] (7,4) node[anchor=north east]{$y$} -- (7,0) -- (12,0) node[anchor=north east]{$x$} ;
	\draw[dashed] (8.75,0)  node[anchor=60]{$\vartheta_{y,x}^-$} -- ++ (0,4);
	\draw[dashed] (10.25,0) node[anchor=120]{$\vartheta_{y,x}^+$} -- ++ (0,4);
	
	\draw[dashed] (7,2.5) node[anchor=east]{$\vartheta_{x,y}^+$} -- ++ (5,0);
	\draw[dashed] (7,1.5) node[anchor=east]{$\vartheta_{x,y}^-$} -- ++ (5,0);	
\end{tikzpicture}
 
\caption{The  phase space for the   the S-system (left), and L-system (right) associated to  \textbf{RN} in Figure~\ref{fig:exampleRN}.}
\end{figure}
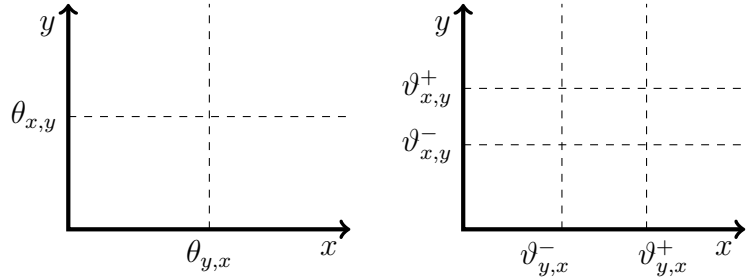

\vspace{12pt}

\begin{defn} \label{faceDefn}
Let $\kappa \in \cK^S$, or $\kappa \in \cK^L$   
where $\kappa := \prod^N_{i=1} I_i$ with $I_i = [\varphi_i,\varphi'_i], \varphi_i \not = \varphi'$,  or half-infinite. Considering $I_j$,
we say that 
\[\tau_j^- := \prod^{j-1}_{i=1} I_i	\times	\{\varphi_j\}	\times	\prod^{N}_{i=j+1} I_i \]
is a \textit{left face} of $\kappa$ with projection index $j$.
Similarly, if $\varphi'_j \not = \infty$, we say that  
\[\tau_j^+ := \prod^{j-1}_{i=1} I_i	\times	\{\varphi'_j\}	\times	\prod^{N}_{i=j+1} I_i \] 
is a \textit{right face} of $\kappa$ with projection index $j$.

A \textit{wall} is a pair $(\tau,\kappa)$, where $\kappa$ is a domain and $\tau$ is a face of $\kappa$. Each wall inherits the projection index from the corresponding face $\tau$ of $\kappa$. We say the \textit{sign of the wall $(\tau,\kappa)$} is 1 (and write $\text{sgn}(\tau,\kappa) = 1$) if $\tau$ is a left face of $\kappa$ and we say the sign of the wall $(\tau,\kappa)$ is $-1$ (and write $\text{sgn}(\tau,\kappa) = -1$) if $\tau$ is a right face of $\kappa$. We denote the collection of walls by $\cW(z)$.
\end{defn}

The goal of this section is to define maps $\cF^S$ and $\cF^L$ on a discrete set of states.
The states will be associated to  domains $\cK^S$ and $\cK^L$. For each pair of domains $\kappa, \kappa'$ such that 
$\tau = \kappa \cap \kappa'$ is a face with some projection index $j$, we will associate a direction, either $+1, -1$, or both, to each pair of walls $(\tau,\kappa)$ and $(\tau,\kappa')$  by making use of the sign of these walls. We then finish by associating an oriented edge (or edges) to the pair of states  that correspond to $\kappa$ and $\kappa'$.

Observe that by definition of functions $\Lambda(x)$ for either S- or L-system (\eqref{generalswitchingsystemequation} and~\eqref{generalLsystemequation}), 
the value of $\Lambda(x)$ is constant on the interior of $\kappa$, $\Int(\kappa)$, where $\kappa \in \cK^S$ for the S-system or $\kappa \in \cK^L_N \subsetneq \cK^L$ for the L-system.
We emphasize  that this is not true for domains $\kappa \in  \cK^L \setminus \cK^L_N $ for L-systems, since there is  at least one variable $x_i \in [\vartheta_{j,i}^-,\vartheta_{j,i}^+]$.

With a slight abuse of the notation we will call this value $\Lambda(\kappa)$.
This value 
\[ \Lambda(\kappa) = ( \Lambda_1(\kappa), \dots, \Lambda_N(\kappa)) \]
has a nice interpretation in terms of solutions of the S-system: all solutions starting at $x \in \mbox{int}(\kappa)$  converge toward the  point 
\[ (\Gamma^S)^{-1}\Lambda^S(\kappa) = ( \Lambda_1^S(\kappa)/\gamma^S_1, \dots, \Lambda_N^S(\kappa)/\gamma^S_N),\]
where $\Gamma^S$ is a diagonal matrix consisting of $\gamma^S_i$.

\vspace{12pt}

\begin{defn}\label{def:focalpt}
	The \textit{focal point} or \textit{target point} of a domain $\kappa \in \cK^S$ at a valid S-parameter  or $\kappa \in \cK^L_N$ at a valid  L-parameter is the value $\Gamma^{-1}\Lambda(\kappa)$, where $\Gamma = \Gamma^S$, or $\Gamma = \Gamma^L$, respectively. To simplify notation  will use notation $\Gamma$ for both as the superscript will be clear from the context.  
	If $\Gamma^{-1}\Lambda(\kappa) \in \kappa'$ we call $\kappa'$ a  \textit{target domain}.  We say that $\kappa$ is an \textit{attracting domain} if $\Gamma^{-1}\Lambda(\kappa) \in \kappa$.
\end{defn}

\vspace{12pt}

  \begin{defn}
A \textit{regular S-parameter} $z^S$ satisfies for all $i = 1,\dots, N$, $j \in \mathbf{T}(i)$, and $\kappa \in \cK^S$, 
\[ \Lambda^S_i(\kappa)/\gamma_i^S \neq \theta_{j,i}.\]
We call the space of all regular S-parameters $Z^S$.
A \textit{regular  L-parameter} $z^L$ satisfies for all $i = 1,\dots, N$, $j \in \mathbf{T}(i)$, and $\kappa \in \cK^L_N$, 
\[ \Lambda_i^L(\kappa)/\gamma_i^L \neq \vartheta_{j,i}^\pm .\]
We call the space of all regular L-parameters $Z^L$.
\end{defn}

\subsection{Nearest neighbor multi-valued map $\cF^S$ for the S-system}

We first define the map $D^S$ which mimics the action of the target point function $\Gamma^{-1}\Lambda: \R^N \to \R^N$ on  the level of states $\cV$. Let 
$D^S: \cV^S \to \cV^S$ by 
\begin{equation}\label{eq:D^S} 
D^S \circ G^S = G^S \circ \Gamma^{-1}\Lambda^S,
\end{equation}
where recall that $G^S$ is defined in Section~\ref{defn:params_S}.
We now define the map $D^S$ more explicitly.
Since  $\Gamma^{-1}\Lambda = \Gamma^{-1} \bar M \circ \sigma$, let 
\[  M := \Gamma^{-1} \bar M .\]
Then since 
\[ G^S \circ \Gamma^{-1}\Lambda^S = G^S \circ \Gamma^{-1} \bar M \circ \sigma^S = G^S \circ M \circ F^S \circ G^S\]
we have 
\[ D^S:= G^S \circ M \circ F^S,\]
analogously to~\eqref{eq:D}.
We now proceed to construct a nearest neighbor multi-valued map $\cF^S$ from the focal points of the S-system~\eqref{generalswitchingsystemequation}, then we show that $\cF^S$ is compatible with the S-system in the sense of Definition~\ref{defn:compatible}, and, finally, we show that $\cF^S$ is an asynchronous update rule for $D^S$.

\vspace{12pt}

\begin{defn}\label{def:Swall_labeling}
Let $z^S$ be a regular parameter for an S-system and  $\cK^S$ the corresponding set of domains. The \textit{wall-labeling} of $\cW(z^S)$ is a function $\scrL^S:\cW(z^S) \to \{-1,1\}$ defined as follows. Let $(\tau,\kappa) \in \cW(z^S)$ be a wall with projection index $j$; i.e. $\tau \subset \{x_i = \theta_{j,i}\}$. Then define 
\[\scrL^S((\tau,\kappa)):= \text{sgn}(\tau,\kappa) \cdot \text{sgn}(\Lambda_i^S(\kappa)/\gamma_i^S - \theta_{j,i}).\]
A wall $(\tau,\kappa)$ is an \textit{absorbing wall} if $\scrL^S((\tau,\kappa)) = -1$ and an \textit{entrance wall} if $\scrL^S((\tau,\kappa)) = 1$.
\end{defn}

Notice that a regular parameter enforces $\Lambda_i(\kappa)/\gamma_i - \theta_{ji} \neq 0$, so that $\scrL^S((\tau,\kappa))$ is defined for all walls.

\vspace{12pt}

\begin{defn}\label{def:Sdomaingraph}

Let $z^S$ be a regular parameter for an S-system, with $\cK^S$ the corresponding set of domains, $\cV^S$ the set of states, and the state function  $G^S$, 
 which  induces  the bijection  $g^S: \cK^S \to\cV^S$ described in Definition~\ref{defn:GFMD}.

 We define a multi-valued function $\cF^S: \cV^S \rightrightarrows \cV^S$ induced by the wall-labeling $\scrL^S$ as follows. Let $g^S(\kappa_1) = s_1$ and $g^S(\kappa_2) = s_2$. Then $s_2 \in \cF^S(s_1)$  if and only if one of the following holds:
		\begin{enumerate}
			\item[(a)]$s_1 = s_2$ and $\kappa$ is an attracting domain.
			\item[(b)] There exists some face $\tau$ such that $(\tau,\kappa_1)$ and $(\tau,\kappa_2)$ are walls and $\scrL^S ((\tau,\kappa_1)) = -1$ (indicating an absorbing wall of $\kappa_1$) and  $\scrL^S ((\tau,\kappa_2)) = 1$ (indicating an entrance wall of $\kappa_2$).
		\end{enumerate}
   As in Definition~\ref{def:asynchronousupdate} we may  represent $\cF$ as a graph $(\cV^S,\cE^S)$, where $(s_1,s_2) \in \cE^S$ if and only if $s_2 \in \cF^S(s_1)$. We call this graph a 
\textit{state transition graph} of the S-system.
\end{defn}

\vspace{12pt}

\begin{thm}\label{thm:FcompatiblewithS}
$\cF^S$ is compatible with the S-system.
\end{thm}

\begin{proof}
The key observation is that all solutions in $\Int(\kappa)$ converge toward the target point $\Gamma^{-1}\Lambda(\kappa)$, while they lie within $\kappa$. 
If there is no trajectory leaving $\kappa$ for an adjacent domain, then the trajectory $x$ must remain within $\kappa$ for all time. Hence the focal point $\Gamma^{-1}\Lambda(\kappa)$ is in $\kappa$, and $\kappa$ is an attracting domain. By Definition~\ref{def:Sdomaingraph} (a), $s \in \cF^S(s)$, where $s=g^S(\kappa)$.

Now assume that there exists a trajectory $x(t) = (x_1(t), \ldots, x_N(t)) $ that passes from $\kappa_1$ to an adjacent $\kappa_2$ via an intervening face $\tau$, where $x_i = \theta_{j,i}$ on $\tau$. First consider the case in which $\tau$ is a right face of $\kappa_1$ and a left face of $\kappa_2$, so that $\text{sgn}(\tau,\kappa_1) = -1$ and $\text{sgn}(\tau,\kappa_2)=1$. 
 Let $x_i(0)  = \theta_{j,i}$. There is an interval $I:= (-\epsilon, \epsilon) \subset \R$ such that $\dot{x}_i(t) >0 $ for all $t \in I\setminus \{0\}$ and $x(t) \in \kappa_1$ for $t \in (-\epsilon,0)$, $x(t) \in \kappa_2$ for $t \in (0,\epsilon)$. Let $t_k = \Lambda^S_i(\kappa_k)/\gamma^S_i$ for $k=1,2$ denote the target points of the domains. Then $\dot x_i(t)  >0$   implies $t_1, t_2 > \theta_{j,i}$, which implies
\[\scrL^S((\tau,\kappa_1))= -1 \cdot 1 \mbox{ and } \scrL^S((\tau,\kappa_2))= 1 \cdot 1 . \]
Thus $s_2 \in \cF^S(s_1)$, where $s_k := g^S(\kappa_k)$, by Definition~\ref{def:Sdomaingraph} (b).

The case in which $\text{sgn}(\tau,\kappa_1) = 1$ and $\text{sgn}(\tau,\kappa_2)=-1$ with $\dot{x}_i(t) <0$ on $I$ is similar.
\end{proof}

The proof of the following theorem we postpone to Appendix~\ref{app:proofSasync} due to length.

\vspace{12pt}

\begin{thm}\label{thm:Sasync}
$\cF^S$ is an asynchronous update rule for $D^S$.
\end{thm}

\subsection{Nearest neighbor multi-valued map $\cF^L$ for the L-system}

Definition~\ref{def:focalpt} defines focal points for all domains $\kappa \in \cK^S$, but only  for domains $\kappa \in \cK^L_N \subsetneq \cK^L$ in the L-system. 
Therefore, we cannot construct a multi-level map $D^L$ analogous to $D^S$ for the S-system, and moreover the wall labeling that we used to construct $\cF^S$ cannot be used in the same way for $\kappa \in \cK^L$ to define $\cF^L$. 
But it turns out that all the information needed to assign directions to all walls in the L-system can be inferred from  just the domains in $\cK^L_N$.  We will use this fact to construct an nearest neighbor multi-valued map $\cF^L$ and associated state transition graph $(\cV^L,\cE^L)$, without going through the map $D^L$ as an intermediary. 

\vspace{12pt}

\begin{defn} \label{cornerpointDefn}
Recall from Definition~\ref{def:cell} that an $\ell$-cell 
has the form $\zeta = \prod^N_{i=1} [\varphi_i ,\varphi'_i]$. 
We define the set of \textit{corner points} of $\zeta$, denoted $\cC(\zeta)$, by
\begin{equation*}
	\cC(\zeta) := \prod^N_{i=1, \varphi'_i \not = \infty} \{\varphi_i ,\varphi'_i\},
\end{equation*}
Note that we exclude $\varphi'_i = \infty$ from the definition of corner points.
\end{defn}

\vspace{12pt}

\begin{defn} \label{def:sgn}
Let $\cC(\zeta)$ be the set of corner points for an $\ell$-cell $\zeta$ for an L-system with regular parameter $z^L$, and let $P, P' \in \cC(\zeta)$ be corner points with $P=(\varphi_1,\varphi_2,\dots,\varphi_N)$ and $P'=(\varphi_1',\varphi_2',\dots,\varphi_N')$. 
We introduce the function $\sgn(\cC(\zeta),k)$ defined as 
\begin{equation}\label{eq:sgn}
\sgn(\cC(\zeta),k)=
\begin{cases}
+1	&	\text{if } \forall P\in \cC(\zeta), \Lambda_k(P)/\gamma_k - \varphi_k>0 \\
-1	&	\text{if } \forall P\in \cC(\zeta), \Lambda_k(P)/\gamma_k - \varphi_k<0 \\
0	&	\text{if } \exists P,P'\in \cC(\zeta) \text{ such that }  \Lambda_k(P)/\gamma_k -\varphi_k>0 \text{ and } \Lambda_k(P') /\gamma_k - \varphi'_k<0\\
\end{cases}
\end{equation}

\end{defn}

Notice that a regular parameter enforces $\Lambda_k^L(P)/\gamma_k -\varphi_k \neq 0$, so that~\eqref{eq:sgn} covers all possible cases. Note also  that since the intervals $ [\vartheta_{k,i}^-, \vartheta_{k,i}^+]$ are disjoint for all $k$, the $i$-th component of every corner point of $\zeta$ is on the boundary of  the set 
$\R \setminus \bigcup_{k} [\vartheta_{k,i}^-, \vartheta_{k,i}^+].$ 
Therefore every corner point of $\zeta$ is also  the corner point of some $\kappa \in \cK^L_N$.

\vspace{12pt}

\begin{defn} \label{def:walllabelingfunction}
Let $z^L\in Z^L$ be a regular parameter. The \textit{wall-labeling} of $\cW(z^L)$ is a function $\scrL^L:\cW(z^L) \to \{-1,0,1\}$ defined as follows. Let $(\tau,\kappa) \in \cW(z^L)$ have a projection index $i$. Then define 
\[\scrL^L((\tau,\kappa)):= \text{sgn}(\tau,\kappa) \cdot \sgn(\cC(\tau), i)\]
A wall $(\tau,\kappa)$ is an \textit{absorbing wall} if $\scrL^L((\tau,\kappa)) = -1$, an \textit{entrance wall} if $\scrL^L((\tau,\kappa)) = 1$, and a \textit{bidirectional wall} if $\scrL^L((\tau,\kappa)) = 0$.
\end{defn}


%
%
%
%

\vspace{12pt}

 \begin{defn}\label{def:Ldomaingraph}
Let $z^L$ be a regular parameter for an L-system, $\cK^L$ be the corresponding  set of domains, $\cV^L$ be the state space and $g^L: \cK^L \to \cV^L$ be the bijection  between the set of domains $\cK^L$ and the states  $\cV^L$.

The nearest neighbor multi-valued map $\cF^L: \cV^L \rightrightarrows \cV^L$  induced by the wall-labeling $\scrL^L$ is defined as follows.  Let $g^L(\kappa_1) = s_1$ and $g^L(\kappa_2) = s_2$. Then $s_2 \in \cF^L(s_2)$  if and only if one of the following holds:
		\begin{enumerate}
			\item[a)]$s_1 = s_2$ and $\kappa \in \cK^L_N$ is an attracting $N$-domain. 
			\item[b)] There exists some face $\tau$ such that $(\tau,\kappa_1)$ and $(\tau,\kappa_2)$ are walls and $\scrL^L ((\tau,\kappa_1)) = -1$ (indicating an absorbing wall of $\kappa_1$) and  $\scrL^L ((\tau,\kappa_2)) = 1$ (indicating an entrance wall of $\kappa_2$).
			\item[c)] There exists some face $\tau$ such that $(\tau,\kappa_1)$ and $(\tau,\kappa_2)$ are walls and $\scrL^L ((\tau,\kappa_1)) = \scrL^L ((\tau,\kappa_2)) = 0$ (indicating a bidirectional wall of both $\kappa_1$ and $\kappa_2$).
		\end{enumerate}
\end{defn}

\vspace{12pt}

\begin{rem}
  We want to make two notes  about this definition.
  \begin{itemize}
  \item First, since attracting domains are only defined for $\kappa \in \cK^L_N \subsetneq \cK^L$ we can only have self-edges on $N$-domains. We justify this choice later by showing that there is always an escape path out of any $\kappa \in \cK^L\setminus \cK^L_N$, see Lemma~\ref{lem:elltoellplus1}.
  \item Second, notice that if there exist two domains $\kappa_1$, $\kappa_2$ in an L-system phase space sharing a face $\tau$, then $\scrL^L ((\tau,\kappa_1)) =0$ if and only if $\sgn(\cC(\tau),i)=0$, which happens  if and only if $\scrL^L ((\tau,\kappa_2)) = 0$. Therefore there is never  a case where $\scrL^L ((\tau,\kappa_1)) = \pm 1$ and $\scrL^L((\tau,\kappa_2)) =0$,  which shows that the assignment of arrows is  always well defined.
  \end{itemize}
  \end{rem}

  \vspace{12pt}

	\begin{thm}\label{thm:Lcompatible} 
	The L-system is compatible with the nearest neighbor multivalued map $\cF^L$. 
	\end{thm}

  \begin{proof}
    If there is no trajectory leaving $\kappa \in \cK^L_N$ for an adjacent domain, then the trajectory $x$ must remain within $\kappa$ for all time. Hence the focal point $\Gamma^{-1}\Lambda^L(\kappa)$ is in $\kappa$, and $\kappa$ is an attracting domain. By Definition~\ref{def:Ldomaingraph} (a), $s \in \cF^S(s)$, where $s=g^S(\kappa)$.

Now assume that there exists a trajectory that passes from $\kappa_1$ to an adjacent $\kappa_2$ via an intervening face $\tau$. First consider the case in which $\tau$ is a right face of $\kappa_1$ and a left face of $\kappa_2$, so that $\text{sgn}(\tau,\kappa_1) = -1$ and $\text{sgn}(\tau,\kappa_2)=1$. Then 
$\dot x_i >0$ on $\tau$ and thus
\[ \sgn(\cC(\tau),i) \in \{0,+1\} \]
since $\sgn(\cC(\tau),i) = -1$ implies $\dot x_i <0$ everywhere on $\tau$ by Theorem~\ref{cornerpointproof}. 
 If $\sgn(\cC(\tau),i) = +1$, then 
\[  \scrL^L((\tau,\kappa_1)) = -1 \cdot 1; \qquad  \scrL^L((\tau,\kappa_1)) = 1 \cdot 1, \]
and $s_2 \in \cF^S(s_1)$, where $s_k := g^S(\kappa_k)$, by Definition~\ref{def:Ldomaingraph} (b). Likewise if $\sgn(\cC(\tau),i) = 0$, then 
\[  \scrL^L((\tau,\kappa_1)) = \scrL^L((\tau,\kappa_1)) = 0, \]
and $s_2 \in \cF^S(s_1)$ by Definition~\ref{def:Ldomaingraph} (c).

The case when $\text{sgn}(\tau,\kappa_1) = 1$ and $\text{sgn}(\tau,\kappa_2)=-1$ is similar.
\end{proof}

\section{Relating $\cF^S$ and $\cF^L$}\label{sec:paramgraph}

We want to compare dynamics of the multi-valued maps $\cF^S$ and $\cF^L$ that are associated to the same network $\textnormal{\textbf{RN}} = (V,E)$. 
Dynamics in our interpretation is characterized by the Morse graph and so our questions will be about the relationship between Morse graphs. Multi-valued maps are also parameterized by parameters $z^S \in Z^S, z^L \in Z^L$ and therefore this comparison must be performed between related parameters. Our goal is to define a canonical map between from the set of regular parameters $Z^S$ to the set of regular parameters $Z^L$. 

We make the following key observations
\begin{itemize}
\item For both S- and L-systems,  the order of the thresholds determines an order of activation (or deactivation)  of  the targets of node $i$ as $x_i$ increases;
\item  the images of the maps  $\cF^S$ and $\cF^L$ depend only on the target \textit{domains}, rather than target points in these domains. Based on these facts,
we define an equivalence relation on the set of regular parameters $z^S,v^S \in Z^S$  and an analogous  equivalence relationship on the set of regular parameters $z^L, v^L \in Z^L$.
\end{itemize}

\vspace{12pt}

\begin{defn}\label{def:orders}
Let \textbf{RN} be a regulatory network.
Let $O =\{O_i\}$ be a collection of orders of the target nodes of node $i$:
\[ O_i = \{ j_1 < j_2 < \cdots < j_{m_i} \;|\; j_k \in \mathbf{T}(i) \}.\]
Let $z$ be a regular parameter of either an S- or L-system. We say that $O$ is the \textit{threshold order} of $z$, and denote it $O(z)$, if $\theta_{j_k,i} < \theta_{j_l,i}$ in $z^S$ or
$[\vartheta^-_{j_k,i},\vartheta^+_{j_k,i}] < [\vartheta^-_{j_l,i},\vartheta^+_{j_l,i}]$ in $z^L$ if and only if $j_k < j_l$ in $O_i$.
\end{defn}

Recall that $\cV^S(i) = \{0,1,2,\dots,m_i\}$ and $\cV^L(i) = \{0,\frac{1}{2}, 1, \frac{3}{2},2,\dots,m_i\}$. 
Note that 
\[  \Psi: \cV^S \hookrightarrow \cV^L \mbox{ by } \Psi(v) = v,\]
  which is a bijection onto its image 
  \[ \cV^{SL} := \Psi(\cV^S).\]
  Since  $\cV^{SL} = g^L(\cK^L_N)$ is a bijection and by the definition of the L-system the target points are defined for $\kappa \in \cK^L_N$, 
   we define a multi-level discrete map $D^L_N$ for the L-system by
  \[ D^L_N \, : \, \cV^{SL} \to \cV^L \mbox{ by }  D^L_N := (g^L)^{-1} \circ  \Gamma^{-1}\Lambda^L(\kappa) \circ g^L \]
  which captures the location of the focal point for each $\kappa \in \cK^L_N$. Note that this multi-level map $D^L_N$ over the subset $\cV^{SL}$ does not uniquely determine a map $D^L$ over $\cV^L$.

\vspace{12pt}

\begin{defn}\label{def:equivalence}
We say that two regular parameters $z^S, v^S \in Z^S$ are are equivalent, 
\[ z^S \stackrel{S}{\sim} v^S \]
if the following hold:
\begin{enumerate}
		\item[(1)]  $O(z^S) = O(v^S)$ and
    \item[(2)] $D^S(z^S) = D^S(v^S)$, where $D^S(z^S)$ and $D^S(v^S)$ are the multi-level discrete maps induced by $z^S$ and $v^S$ respectively.
	\end{enumerate}

We denote the equivalence classes of $\stackrel{S}{\sim}$ by $\cZ^S$, and say $\omega^S \in \cZ^S$. 
  We shall use the notation $O(\omega^S)$ to denote the constant threshold order for all $z^S \in \omega^S$, and use $D^S(\omega^S)$ to denote the fixed multi-level discrete map valid for all parameters in this equivalence class.

The equivalence relationship $\stackrel{L}{\sim}$ between  $z^L, v^L \in \cZ^L$ is defined analogously using the map $D^L_N$. Notation   $\omega^L \in \cZ^L$, $O(\omega^L)$, and $D^L_N(\omega^L)$ is also analogous.
\end{defn}

\vspace{12pt}

\begin{rem} \label{ind}
 Note that the maps $\cF^S$ and $\cF^L$ only depend on the equivalence class $\omega^S$ and $\omega^L$, respectively, and not on individual parameters 
$z^S \in \omega^S$, or $z^L \in \omega^L$. This is because the wall-labeling functions $\scrL^S$ and $\scrL^L$ depend only the location of target points $\Gamma^{-1}\Lambda(\kappa)$ with respect to thresholds. Therefore the Morse graph $\sMG$, which we view as a summary of dynamics of $\cF$,  is also only a function of the equivalence class $\omega$. 
With a slight abuse of notation we will denote by $\cF^S(\omega^S)$ the map $\cF^S$ that corresponds to any parameter 
$z^S \in \omega^S$. Similar notation will be used for the map $\cF^L$. 
\end{rem}

As an aside, note that the elements of $\cZ^S$
correspond to the nodes of a {\it combinatorial parameter graph}, introduced in~\cite{us2}. The edges in this graph correspond to a single change in the order of indices in $O_i$ for a single $i$, or to a change in the $i$-th component from one target state to an adjacent target state for a single $i$.

\vspace{12pt}

\begin{defn}\label{def:canonicalperturbation}
	The \textit{canonical map} $\Omega: \cZ^S \hookrightarrow \cZ^L$ maps $\omega^S\in \cZ^S $ to $\omega^L\in \cZ^L$ if and only if 	
	\begin{enumerate}
		\item $O(\omega^S) = O(\omega^L)$ and
		\item $D^S(\omega^S) = D^L_N(\omega^L)$.  
	\end{enumerate}
\end{defn}
It is easy to see that for a fixed network \textbf{RN} the map $\Omega$ is injective.

At this point we are ready to precisely formulate a central question of this paper. For a fixed network \textbf{RN} we have introduced two different classes of multi-valued maps $\cF^S(\omega^S)$ and $\cF^L(\omega^L)$ motivated by S- and L-systems of differential equations. These are both valid choices of a model that captures the dynamics of the network, and they both describe a  family of dynamical models that depend in a continuous way on a high dimensional set of parameters.  In Remark~\ref{ind} we have shown that the number of dynamical behaviors, as captured by the Morse graphs, is finite, as it only depend on the equivalence class $\omega \in \cZ$, and $\cZ$ is a finite set.
Finally, in~(\ref{def:canonicalperturbation}) we have shown that there is a bijection $\Omega$ between $\cZ^S$ and a subset of $\cZ^L$.  A  natural question is, what is the relationship between the dynamics of $\cF^S(\omega)$ and $\cF^L(\Omega(\omega))$?
More precisely, what is the correspondence between the Morse graphs of $\cF^S(\omega)$ and $\cF^L(\Omega(\omega))$? 
The rest of the paper is devoted to this question.

\section{Morse graphs  of  $\cF^S(\omega)$ and $\cF^L(\Omega(\omega))$}\label{sec:Morsegraphs}

The results in this section depend on the absence of negative self-regulation, which we assumed in the definition of a regulatory network at the beginning of this work. In the Appendix, we prove two important, but technical, lemmas that are the basis of the theorems in this section.
\begin{itemize}
  \item Lemma~\ref{lem:equivpath} states that an edge $v \to v'$ in the state transition graph of the S-system exists if and only if the path $\Psi(v) \to u \to \Psi(v')$ exists in the L-system state transition graph, where $u = g^L(\eta)$ with  $\eta \in \cK^L_{N-1}$ is uniquely defined. This  implies that every path that exists in  $\cV^S$ can be lifted to a path  in $\cV^L$ by  adding intermediate nodes $u$ that correspond to domains $\eta := (g^L)^{-1}(u)$ that belong to $ \cK^L_{N-1}$. This leads immediately to Corollary~\ref{cor:2npath}, which extends the result to a path of any length in the graph $(\cV^S,\cE^S)$ generated by $\cF^S$. In other words, paths in $(\cV^S,\cE^S)$ are equivalent to a select set of paths in $(\cV^L,\cE^L)$.
  \item Lemma~\ref{lem:elltoellplus1} proves that if $w \in \cV^L$ has $n>0$ non-integer values, then there exists $w' \in \cV^L$ such that $(w,w') \in \cE^L$ and $w'$ has $n-1$ non-integer values. Therefore for any state $w \in \cV^L$, there is a path from $w$ to $w''$ with  $(g^L)^{-1}(w'') \in \cK^L_N$. In other words, every node in $\cV^L$ has a path to a node in $\cV^{SL}$. This is our justification for defining self-edges in the L-system state transition graph based only on states corresponding to $\kappa \in \cK^L_N$. \end{itemize}

We now proceed to the main theorems of the section that describe the characteristics of a map between the Morse graphs of $\cF^S(\omega)$ and $\cF^L(\Omega(\omega))$.

\vspace{12pt}

  \begin{defn} 
  A map $f : X \to Y$ between ordered spaces $(X, \leq)$ and $(Y, \leq)$  is \textit{order preserving} if $x_1 < x_2$ implies $f(x_1) \leq f(x_2)$.
  \end{defn}

  \vspace{12pt}

  \begin{defn}
  Referring back to Definition~\ref{defn:morsegraph}, let $U,V \in \sMD$ be two sets in a Morse decomposition. We define the order $U \preceq V$ only if there exists a path from an element $v \in V$ to an element $u \in U$ in the associated state transition graph of $\cF$. The inequality is strict if there is no return path. 
\end{defn}

\vspace{12pt}

\begin{thm}\label{thm:ordpreserv}

Consider  $\omega^S \in \cZ^S$ and   the nearest neighbor multi-valued maps $\cF^S(\omega^S)$ and $\cF^L(\Omega(\omega^S))$. Consider the associated Morse graphs $\sMD^S$ and $\sMD^L$.
	Then there is  order-preserving map
	\[ \phi: \mathsf{MD}^S \to \mathsf{MD}^L,\]
  defined by $ C':=\phi(C)$ where $C'$ is the smallest (under inclusion)  element  of $\sMD^L$ containing the states 
  $\Psi(C) := \{ \Psi(v) \:|\; v \in C\}$. 
\end{thm}

\begin{proof}
Since $U,V \in \sMD^S$, they are strongly connected components of the graph induced by $\cF^S$, there a path between any pair of nodes in $U$ and any pair of nodes in $V$. 
 By Corollary~\ref{cor:2npath} the paths between the nodes in $U$ and the nodes of $V$ lift to paths  in the graph of $\cF^L$. Therefore the sets $\Psi(U)$ and $\Psi(V)$ are also strongly connected. Therefore  $\Psi(U)$ and $\Psi(V)$ must  be subsets of Morse sets in $\sMD^L$, say $U',V' \in \sMD^L$, respectively.  We  define $U' := \phi(U)$ and $V' := \phi(V)$.  
 
 Finally, if $U  \preceq V$ in  $ \mathsf{MD}^S$, there must be a path  from an element $v \in V$ to an element $u \in U$. By  Corollary~\ref{cor:2npath} there is a path from $\Psi(v)$  to $\Psi(u)$ in the graph of $\cF^L$ and therefore  $U' \preceq V'$ in $\sMD^L$.  This finishes the proof.
\end{proof}

Notice that in the proof of Theorem~\ref{thm:ordpreserv}, we cannot conclude that if $U \prec V$ then  $U' \prec V'$, only that $U' \preceq V'$. As we will see later, it may be that  $U \prec V$ but $V' = U'$. Theorem~\ref{thm:ordpreserv} shows that the general ordering of Morse sets remains similar between the two systems. However, the property of order preservation is not very strong and we will examine ways that we can strengthen this result, as well as reasons why this fails in general case.

\vspace{12pt}

\begin{defn}  An attractor in a Morse graph $\mathsf{MG}$ is a minimal element in the partial order, represented by its Morse set $A \in \sMD$.
The collection of all attractors of $\mathsf{MD}^*$ of the map $\cF^*(\omega)$ at $\omega \in \cZ^*$ will be denoted by $\mathcal{A}^*$
for both $*=S,L$.  We will call all Morse sets that are not attractors  \textit{unstable Morse sets.}
\end{defn}

Note that the minimality of the attractor in the partial order implies that  
\begin{itemize}
  \item Every forward path starting in an attractor remains in the attractor.
  \item Every state has a forward path to an attractor.
\end{itemize}

\vspace{12pt}

  \begin{thm}\label{thm:surjection}
  Consider a network \textbf{RN} with associated multi-valued maps $\cF^S(\omega^S)$ and $\cF^L(\Omega(\omega^S))$, and associated 
  Morse decompositions  $\mathsf{MD}^S$ and  $\mathsf{MD}^L$.
  Then the order preserving map $\phi: \mathsf{MD}^S \to \mathsf{MD}^L$ restricts to a surjection over the sets of attractors
  \[  \phi: \mathcal{A}^S \stackrel{onto}\longrightarrow \mathcal{A}^L. \]
  \end{thm}

  \begin{proof}
  Let   $A^S \in \mathcal{A}^S \subset \cV^S$ be an attractor in $\mathsf{MD}^S$. 
    Since $A^S$ is a strongly connected subgraph, for any $v,v' \in A^S$ there exists a path $v \to \cdots \to v'$ in the state transition graph of $\cF^S$. 
    Let $w:=\Psi(v)$ and $w':=\Psi(v')$.  By Corollary~\ref{cor:2npath}, there is then a   path from $w$ to $w'$ in the state transition graph of $\cF^L$. 
    Let 
    \[ U  := \{w' \in \cV^L \;|\; \mbox{ there is a path from any }  w \in \Psi(A^S) \mbox{ to } w'  \in \cV^L\} .\]
    and let $A$ be the maximal strongly connected subgraph containing $U$. Clearly, $A$ is an attractor in $\cA^L$. It follows  from the definition of $\phi$ that
    \[ \phi(A^S) = A.\]
    We now need to show that the restriction of $\phi$ on  $\mathcal{A}^S$ is a surjection. 
    
    Consider an arbitrary attractor $A^L \subset \cA^L$. If $u \in A^L$ such that $(g^L)^{-1} (u) \not \in \cK^L_N$, by Lemma~\ref{lem:elltoellplus1} there is a  path from $u$ to $u'$ with $u' = g^L(\kappa)$ for some $\kappa \in \cK^L_N$.  Since forward paths starting at $u \in A^L$ remain in $A^L$ by definition of the attractor, we conclude that $A^L$ must contain the vertex $u'$.
    To prove $\phi$ is surjective, we assume by contradiction that there is $A^L \subset \cA^L$ such that 
 $A^L \not = \phi(A^S)$ for any attractor  $A^S \in \cA^S$.
   By the argument above, there is a vertex $u' \in A^L$ such that  $u' = g^L(\kappa)$ for some $\kappa \in \cK^L_N$. Let  $\Psi^{-1}(u') =: v' \in \cV^S$. 
  There must be  an attractor $A_1^S$ and $v \in  A_1^S$ such that there is a path from $v'$ to $v$ in the state transition graph of $\cF^S$. By  Corollary~\ref{cor:2npath} this path can be lifted to a path from  $u'$ to $w:=\Psi(v)$ in the state transition graph of $\cF^L$.  Since $u' \in A^L$ it must be that $w \in A^L$ and by construction of  the map $\phi$, 
  $w \in A_1^L := \phi(A_1^S)$. However, since any attractor must contain all of its forward paths and $u' \in A^L$, this implies that $A^L \supseteq A_1^L$. Since $A_1^L$ is maximal by construction, $A^L = A_1^L = \phi(A_1^S)$.  This contradicts  the assumption that $A^L \not = \phi(A^S)$, which proves the theorem.
    \end{proof}

Theorem~\ref{thm:surjection} says the the long term behavior of $\cF^S$ captures all long time behavior of the richer map $\cF^L$  at the corresponding parameters for networks with no negative self-regulation. Attractors of $\cF^S$ can disappear under this correspondence $\Omega$ but they cannot be created when employing the class of models  $\cF^L$.  The assumption that \textbf{RN} does not contain a negative self-loop is essential here. In fact that result is not true when there is negative self-regulation~\cite{Ironi2011}, because then fixed points may appear in domains $\kappa \in \cK^L \setminus \cK^L_N$.

%

We show a  stronger relationship between the specific types of attractors of $\cF^S$ and $\cF^L$.

\vspace{12pt}

\begin{defn}\label{FP}
Let $\cA^*$ be the set of attractors associated to a map $\cF^*(\omega)$, where $\omega \in \cZ^*$ and $*=S,L$. 
Define $\sFP^* \subset \cA^*$ to be the set of attractors that consist of a single vertex $v \in \cV^*$. We call $A \in \sFP^*$ a \textit{fixed point}.
\end{defn}

\vspace{12pt}

\begin{thm}\label{thm:FP}
 Consider a network $\mathbf{RN}$ with associated multi-valued maps $\cF^S(\omega^S)$ and $\cF^L(\Omega(\omega^S))$. 
   Let $\sFP^S$ and $\sFP^L$ be the associated sets of fixed points.
  Then the restriction of $\phi$ to $\sFP^S$ is a bijection 
  \[ \phi : \sFP^S \to \sFP^L. \]
\end{thm}

\begin{proof}
Consider $w \in \sFP^L$. 
Then by Lemma~\ref{lem:elltoellplus1} the domain $\kappa := (g^L)^{-1}(w) $ must be an N-domain $\kappa \in \cK_N^L$, and therefore $v:=\Psi^{-1}(w) \in \cV^S$.
Then it follows from  Corollary~\ref{cor:2npath} that since there are no paths that exit $w$ in $\cV^L$, there are no exiting paths from $v$ in $\cV^S$. 
Therefore $v \in \sFP^S$.

\end{proof}

We now present important examples that show that these results cannot be strengthened in several natural directions.  Assume in all statements below that 
$\eta= \Omega(\omega)$. We will show that 
\begin{enumerate}
  \item The map $\phi: \mathsf{MD}^S(\omega) \to \mathsf{MD}^L(\eta)$ does not have to be  surjective, see Lemma~\ref{lem:notsurjective}.
  \item The map $\phi:  \mathsf{MD}^S(\omega) \to \mathsf{MD}^L(\eta)$ does not have to be   injective, see Lemma~\ref{lem:phinotinjectiveingeneral}.
  \item The map  $\phi:  \mathcal{A}^S(\omega) \to \mathcal{A}^L(\eta) $ does not have to be injective, see Lemma~\ref{lem:phinotinjectiveonattractor}.
\end{enumerate}


\section{Examples}


We will now present a series of examples illustrating differences which can arise between $\cF^S$ and $\cF^L$. In all of the following examples, for clarity we suppose that  parameterS $\gamma_i^S = \gamma_i^L = 1$ for all $i \in \{1, \dots ,N \}$. 

We will be comparing the Morse graph of the S-system to the Morse graph of the L-system under the canonical parameter map $\Omega$. Each equivalence class  $\omega \in \cZ^S$, or $\Omega(\omega) \in \cZ^L$ is represented by a collection of inequalities. The inequalities determine the order of the thresholds and the images of maps $D^S$ and $D^L_N$ for every state simultaneously. For example, suppose we have the nodes $i,j,k,m,n \in \textbf{RN}$, with $i,j$ additive, positive inputs to  node $k$, and $m,n$ the outputs of $k$. Assume that $i$ and $j$ only affect node $k$, so that they each have two states, 0 and 1. Since $k$ has two output thresholds, it has three states, 0, 1, and 2. Suppose further that we have the inequality description
\begin{equation} \label{k:ineq}
 l_{k,i} + l_{k,j} < \theta_{m,k} < u_{k,i} + l_{k,j} < \theta_{n,k} <  l_{k,i} + u_{k,j} < u_{k,i} + u_{k,j} 
 \end{equation}
for the node $k$. This implies that the threshold order of $k$ is 
\[ O_k = \{  m < n \} \]
and that the $k$-th component of the map $D^S$ is 
\begin{align*}
  00 \mapsto 0, \quad 10 \mapsto 1, \quad 01 \mapsto 2, \quad 11 \mapsto 2
\end{align*}
This happens because when $i,j$ are in their 0 states, they are contributing a low value $l$, and when they are in their higher 1 states, they are contributing a high value $u$ to node $k$. 
A complete description of $\omega \in \cZ^S$ includes inequality description like (\ref{k:ineq}) for every node $k \in  \mathbf{RN}$.

In this section all the explicit examples of differential equations will be those of S-systems. We therefore drop the superscript $S$ on functions $\sigma^S$, and will use 
$\sigma^-$ and $\sigma^+$ to denote piecewise constant nonlinearities that correspond to negative versus positive regulation in \textbf{RN}, respectively. See Figure~\ref{fig:exampleSandLsystem} (left) for an example of $\sigma^-$. In $\sigma^+$, the lower constant value would occur first.

\vspace{12pt}

\begin{lem}\label{lem:notsurjective}
The map $\phi : \sMD^S \to \sMD^L$ is, in general,  not surjective.
\end{lem}

\begin{proof}
Consider the network shown in Figure~\ref{fig:exampleRN}, with system of equations 
\begin{equation}
\dot x = -\gamma_xx + \sigma^-_{x,y}(y), \qquad  \dot y = -\gamma_yy + \sigma^-_{y,x}(x)
\end{equation}
and $\omega \in \cZ^S$ satisfying 
\begin{equation}
l_{x,y} < \theta_{y,x} < u_{x,y}, \qquad l_{y,x} < \theta_{x,y} < u_{y,x}.
\end{equation}
The corresponding graphs of $\cF^S(\omega)$ and $\cF^L(\Omega(\omega))$ are shown in Figure~\ref{fig:bistabledomaingraph}, along with the corresponding Morse graphs. It is clear that there cannot exist a surjective map from $\sMD^S$ to $\sMD^L$, since $\sMD^L$ has an extra unstable Morse set.
\end{proof}

In Figure~\ref{fig:bistabledomaingraph}, the nodes in the Morse graphs labeled FP are fixed points in the set $\sFP$. These correspond in both state transition graph (STG) to states with the self loops. The  Morse set labeled FC is an unstable Morse set composed of all of the nodes that are not in the corners of STG. Thus Morse  set FC is generated by the strongly connected four-leaf clover structure in the middle of the STG, and is clearly unstable since it has paths to the fixed points. Since it is composed of more than one node this Morse set is consistent with  cyclic behavior in STG. Since there is cyclic behavior for all (both) variables of the system, we will call this type of Morse set a ``full cycle'', or FC. When there is a Morse set with cyclic behavior in  a subset of all variables, we will label it XC.

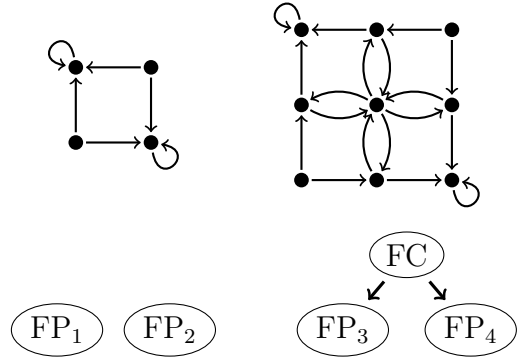
\begin{figure}
\centering
\begin{tikzpicture}[scale=1]
      \node [circle,fill,scale=.5] (bl) at (0,0.5) {};
      \node [circle,fill,scale=.5] (br)at (1,0.5) {};
      \node [circle,fill,scale=.5] (tr) at (1,1.5) {};
      \node [circle,fill,scale=.5] (tl) at (0,1.5) {};
       
      \draw [shorten >= 1pt,shorten <= 1pt, thick,->] (bl) to (tl);
      \draw [shorten >= 1pt,shorten <= 1pt,thick,->] (bl) to (br);
      \draw [shorten >= 1pt,shorten <= 1pt,thick,->] (tr) to (tl);
      \draw [shorten >= 1pt,shorten <= 1pt,thick,->] (tr) to (br);
      \draw [shorten >= 1pt,shorten <= 1pt,thick,->] (tl) to[out=100,in=170, looseness=12] (tl);
      \draw [shorten >= 1pt,shorten <= 1pt,thick,->] (br) to[out=280,in=350,looseness=12] (br);
  
      \node [circle,fill,scale=.5] (1) at (3,0) {};
      \node [circle,fill,scale=.5] (2)at (4,0) {};
      \node [circle,fill,scale=.5] (3) at (5,0) {};
      \node [circle,fill,scale=.5] (4) at (3,1) {};     
      \node [circle,fill,scale=.5] (5) at (4,1) {}; 
      \node [circle,fill,scale=.5] (6) at (5,1) {}; 
      \node [circle,fill,scale=.5] (7) at (3,2) {}; 
      \node [circle,fill,scale=.5] (8) at (4,2) {}; 
      \node [circle,fill,scale=.5] (9) at (5,2) {}; 

        \draw [shorten >= 1pt,shorten <= 1pt,thick,->] (1) to (4);
      \draw [shorten >= 1pt,shorten <= 1pt,thick,->] (4) to (7);
      \draw [shorten >= 1pt,shorten <= 1pt,thick,->] (1) to (2);
      \draw [shorten >= 1pt,shorten <= 1pt,thick,->] (2) to (3);
      \draw [shorten >= 1pt,shorten <= 1pt,thick,->] (9) to (8);
      \draw [shorten >= 1pt,shorten <= 1pt,thick,->] (8) to (7);
      \draw [shorten >= 1pt,shorten <= 1pt,thick,->] (9) to (6);
      \draw [shorten >= 1pt,shorten <= 1pt,thick,->] (6) to (3);
      \draw [shorten >= 1pt,shorten <= 1pt,thick,->] (8) to[bend left] (5);
      \draw [shorten >= 1pt,shorten <= 1pt,thick,<-] (8) to[bend right] (5);
      \draw [shorten >= 1pt,shorten <= 1pt,thick,->] (5) to[bend left] (2);
      \draw [shorten >= 1pt,shorten <= 1pt,thick,<-] (5) to[bend right] (2);
      \draw [shorten >= 1pt,shorten <= 1pt,thick,<-] (4) to[bend left] (5);
      \draw [shorten >= 1pt,shorten <= 1pt,thick,->] (4) to[bend right] (5);
      \draw [shorten >= 1pt,shorten <= 1pt,thick,<-] (5) to[bend left] (6);
      \draw [shorten >= 1pt,shorten <= 1pt,thick,->] (5) to[bend right] (6);
      \draw [shorten >= 1pt,shorten <= 1pt,thick,->] (7) to[out=100,in=170, looseness=12] (7);
      \draw [shorten >= 1pt,shorten <= 1pt,thick,->] (3) to[out=280,in=350,looseness=12] (3);
  \end{tikzpicture}
  
\begin{tikzpicture}[main node/.style={inner sep=2pt,ellipse,fill=white!20,draw}]
    \node[main node] (FP1) at (0,0) {$\text{FP}_1$};
    \node[main node] (FP2) at (1.5,0) {$\text{FP}_2$};
    
    \node[main node] (FP3) at (3.8,0) {$\text{FP}_3$};
    \node[main node] (FP4) at (5.5,0) {$\text{FP}_4$};
    \node[main node] (FC) at (4.65,1) {$\text{FC}$};

    \path[->,thick]
    (FC) edge[shorten >= 3pt,shorten <= 3pt, very thick] node[] {} (FP3)
    (FC) edge[shorten >= 3pt,shorten <= 3pt, very thick] node[] {} (FP4);
\end{tikzpicture}  
  \caption{The corresponding $\cF^S(\omega)$ (top left) and $\cF^L(\Omega(\omega))$ (top right) domain graphs of the bistable example network under a canonical map. Also shown are the Morse graphs of $\cF^S(\omega)$ (bottom left) and $\cF^L(\Omega(\omega))$ (bottom right). }
  \label{fig:bistabledomaingraph}
\end{figure}

Though Lemma~\ref{lem:equivpath} does guarantee that  each path in the graph of $\cF^S$ has a  corresponding path in the graph of $\cF^L$, the converse is not necessarily true.  This insight is key to understanding the limits on the relationship induced by the map $\phi$. 
To illustrate this fact, we will construct the following examples in multiple steps. We first start from simple two dimensional networks and the successively embed this example to more and more complicated examples all the way to  five-dimensional networks. The lower dimensional examples establish how paths can exist in the graph of $\cF^L$ that have no correlate in $\cF^S$.

First consider the (partial) network $z \dashv y$ with $l_{y,z} < \theta_{*,y} < u_{y,z}$ where $\theta_{*,y}$ is an unspecified threshold of $y$, and suppose that $\dot z < 0$. Notice that this is only part of a network in the sense that we would need at least one other interaction in order to define the threshold of $y$. This will be resolved when we embed this interaction into a three-dimensional network. The S-system state transition graph is shown in Figure \ref{fig:2dfundamentalstructure} on the left, and the one for the L-system on the right, where recall $\Psi : \cV^S \to \cV^{SL}$ is the map between states in the S- and L-systems. 
Notice that paths between $v_i$ and $v_j$ that exist in the S-system graph map to paths between $\tilde v_i := \Psi(v_i)$ and $\tilde v_j := \Psi(v_j)$ with intermediate nodes corresponding to (N-1)-domains (see Lemma \ref{lem:equivpath}). However, there is now a path from $\tilde v_3$ to $\tilde v_2$ which does not contain $\tilde v_1$, shown in cyan. The structure of the graph shown in Figure \ref{fig:2dfundamentalstructure} on the right will play a key role in all later examples. We will embed this structure into higher dimensional state transition graphs.

\begin{figure}[h!]
\centering
\begin{minipage}{.5\textwidth}
  \centering
\begin{tikzpicture}[main node/.style={inner sep=3pt,circle,fill=white!20,draw,font=\sffamily\normalsize\bfseries}, scale=1.25]
    \node[main node] (1) at (0,0) {$v_1$};
    \node[main node] (2) at (1,0) {$v_2$};
    \node[main node] (3) at (0,1) {$v_3$};
    \node[main node] (4) at (1,1) {$v_4$};
    
    \path[->,>=angle 90,thick]
    (1) edge[shorten >= 3pt,shorten <= 3pt] node[] {} (2)
    (4) edge[shorten >= 3pt,shorten <= 3pt] node[] {} (3)
    (3) edge[shorten >= 3pt,shorten <= 3pt] node[] {} (1)
    (4) edge[shorten >= 3pt,shorten <= 3pt] node[] {} (2)
    ;
    
    \draw[very thick,->] (-1,-.4) -- +(.6,0) node[anchor=north] {$y$} ;
    \draw[very thick,->] (-1,-.4) -- +(0,.6) node[anchor=west]  {$z$} ;   
  \end{tikzpicture} 
\end{minipage}%
\begin{minipage}{.5\textwidth}
  \centering
  \begin{tikzpicture}[main node/.style={inner sep=3pt,circle, fill=white!20,draw,font=\sffamily\normalsize\bfseries}, scale=1.25]
    \node[main node] (5) at (0,0) {$\tilde v_1$};
    \node[main node] (56) at (1,0) {};
    \node[main node] (6) at (2,0) {$\tilde v_2$};
    \node[main node] (57) at (0,1) {};
    \node[main node] (5678) at (1,1) {};
    \node[main node] (68) at (2,1) {}; 
    \node[main node] (7) at (0,2) {$\tilde v_3$};
    \node[main node] (78) at (1,2) {};
    \node[main node] (8) at (2,2) {$\tilde v_4$};
    
    \path[->,>=angle 90,thick]
    (8) edge[shorten >= 3pt,shorten <= 3pt] node[] {} (78) 
    (78) edge[shorten >= 3pt,shorten <= 3pt] node[] {} (7)
    (5) edge[shorten >= 3pt,shorten <= 3pt] node[] {} (56) 
    (56) edge[shorten >= 3pt,shorten <= 3pt] node[] {} (6)
    
    (5678) edge[shorten >= 3pt,shorten <= 3pt,->,bend right=20] node[] {} (57)
    (57) edge[cyan, shorten >= 3pt,shorten <= 3pt,->,bend right=20] node[] {} (5678)
    (68) edge[shorten >= 3pt,shorten <= 3pt,->,bend right=20] node[] {} (5678)
    (5678) edge[cyan, shorten >= 3pt,shorten <= 3pt,->,bend right=20] node[] {} (68)

    (7) edge[cyan, shorten >= 3pt,shorten <= 3pt] node[] {} (57)
    (57) edge[shorten >= 3pt,shorten <= 3pt] node[] {} (5)
    (8) edge[shorten >= 3pt,shorten <= 3pt] node[] {} (68)
    (68) edge[cyan, shorten >= 3pt,shorten <= 3pt] node[] {} (6)
    (78) edge[shorten >= 3pt,shorten <= 3pt] node[] {} (5678)
    (5678) edge[shorten >= 3pt,shorten <= 3pt] node[] {} (56)     
    ;
    
    \end{tikzpicture}
\end{minipage}  
\caption{The fundamental structures of $\cF^S(\omega)$, shown left, and $\cF^L(\Omega(\omega))$, shown right, where $\tilde v_i := \Psi(v_i)$. Note the path from $\tilde v_3$ to $\tilde v_2$, shown in cyan, does not contain $\tilde v_1$. This structure will be embedded in all later examples.}
\label{fig:2dfundamentalstructure}
\end{figure}
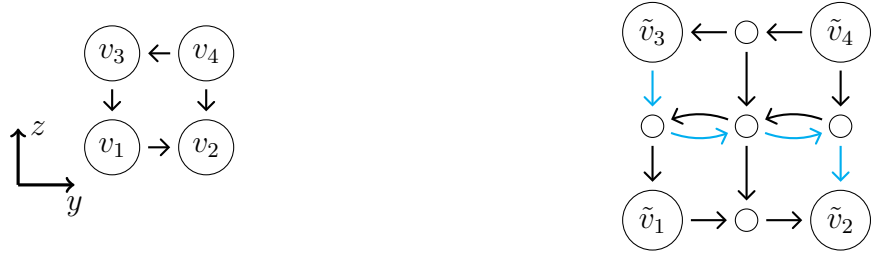

\vspace{12pt}

\begin{lem} \label{lem:3dpath}
Consider the graphs for $\cF^S(\omega)$ and $\cF^L(\Omega(\omega))$. A path from $\Psi(v_i)$ to $\Psi(v_j)$ in the graph of $\cF^L$ 
does not guarantee the existence of a path from $v_i$ to $v_j$ in the graph of $\cF^S$.
\end{lem}

\begin{proof}
We will now embed the previous partial network in a 3-dimensional  network, given by $z \dashv y \to x \dashv z$, and let $\omega \in \cZ^S$ satisfy
\begin{equation} \label{parameter3d}
l_{y,z} < \theta_{x,y} < u_{y,z}, \qquad l_{x,y} < \theta_{z,x} < u_{x,y} \qquad  l_{z,x} < \theta_{y,z} < u_{z,x}.
\end{equation}
Then $\cF^S(\omega)$ constructed by the wall-labeling function $\scrL^S$ (see~\eqref{parameter3d}) is given in Figure~\ref{fig:3dpath} on the left. Notice that there does not exist a path from $u$ to $u'$. Also note that the structure from Figure~\ref{fig:2dfundamentalstructure} (left) occurs both on the front face and the back face of the cube in Figure~\ref{fig:3dpath} (left). The corresponding $\cF^L(\Omega(\omega))$ is shown on the right with a few nodes removed for visual clarity.
The bidirectional arrows in Figure~\ref{fig:2dfundamentalstructure} (right) correspond to bidirectional arrows in Figure~\ref{fig:3dpath} (right) that allow us to find a path from $\tilde u$ to $\tilde u'$ in $\cF^L(\Omega(\omega))$, where $\tilde u := \Psi (u)$. 
\end{proof}

\begin{figure}[h!]
\centering
\begin{minipage}{.5\textwidth}
  \centering
\begin{tikzpicture}[main node/.style={inner sep=0pt,minimum size=8pt,circle,fill=white!20,draw,font=\sffamily\normalsize\bfseries}, scale=2]
    \node[main node,inner sep=2pt] (1) at (0,0,1) {$u'$};
    \node[main node] (2) at (1,0,1) {};
    \node[main node] (3) at (0,1,1) {};
    \node[main node] (4) at (1,1,1) {};
    \node[main node] (5) at (0,0,0) {}; 
    \node[main node] (6) at (1,0,0) {};
    \node[main node,inner sep=2pt] (7) at (0,1,0) {$u$};
    \node[main node] (8) at (1,1,0) {};
    
    \path[->,>=angle 90,thick]
    (1) edge[shorten >= 3pt,shorten <= 3pt] node[] {} (2)
    (5) edge[shorten >= 3pt,shorten <= 3pt] node[] {} (6)
    (4) edge[shorten >= 3pt,shorten <= 3pt] node[] {} (3)
    (8) edge[shorten >= 3pt,shorten <= 3pt] node[] {} (7)
    (3) edge[shorten >= 3pt,shorten <= 3pt] node[] {} (1)
    (4) edge[shorten >= 3pt,shorten <= 3pt] node[] {} (2)
    (7) edge[shorten >= 3pt,shorten <= 3pt] node[] {} (5)
    (8) edge[shorten >= 3pt,shorten <= 3pt] node[] {} (6)
    (1) edge[shorten >= 3pt,shorten <= 3pt] node[] {} (5)
    (6) edge[shorten >= 3pt,shorten <= 3pt] node[] {} (2)
    (3) edge[shorten >= 3pt,shorten <= 3pt] node[] {} (7)
    (8) edge[shorten >= 3pt,shorten <= 3pt] node[] {} (4)
    ;
    
    \draw[very thick,->] (-1.5,-1,-1) -- ++(.4,0,0) node[anchor=north] {$y$} ;
    \draw[very thick,->] (-1.5,-1,-1) -- ++(0,.4,0) node[anchor=west]  {$z$} ;
    \draw[very thick,->] (-1.5,-1,-1) -- ++(0,0,.4) node[anchor=north]  {$x$} ;     
\end{tikzpicture}
\end{minipage}%
\begin{minipage}{.5\textwidth}
  \centering
\begin{tikzpicture}[main node/.style={inner sep=0pt,minimum size=8pt,circle,fill=white!20,draw,font=\sffamily\normalsize\bfseries}, scale=3]  
    \node[main node] (a) at (0,.5,0) {};
    \node[main node] (b) at (.5,.5,0) {};
    \node[main node] (c) at (1,.5,0) {};
    
    \node[main node,inner sep=2pt] (1) at (0,0,1) {$\tilde u'$};
    \node[main node] (2) at (1,0,1) {};
    \node[main node] (3) at (0,1,1) {};
    \node[main node] (4) at (1,1,1) {};
    \node[main node] (5) at (0,0,0) {}; 
    \node[main node] (6) at (1,0,0) {};
    \node[main node,inner sep=2pt] (7) at (0,1,0) {$\tilde u$};
    \node[main node] (8) at (1,1,0) {};

    \node[main node] (d) at (1,.5,.5) {};
    \node[main node] (e) at (1,.5,1) {};
    \node[main node] (f) at (.5,.5,1) {};
    \node[main node] (g) at (0,.5,1) {};    
      
    \path[->,>=angle 90,thick]
  
    (3) edge[shorten >= 1pt,shorten <= 1pt] node[] {} (g)
    (g) edge[cyan, shorten >= 1pt,shorten <= 1pt] node[] {} (1)
    (4) edge[shorten >= 1pt,shorten <= 1pt] node[] {} (e)
    (e) edge[shorten >= 1pt,shorten <= 1pt] node[] {} (2)   
    (7) edge[cyan, shorten >= 1pt,shorten <= 1pt] node[] {} (a)
    (a) edge[shorten >= 1pt,shorten <= 1pt] node[] {} (5)     
    (8) edge[shorten >= 1pt,shorten <= 1pt] node[] {} (c)
    (c) edge[shorten >= 1pt,shorten <= 1pt] node[] {} (6)     
    (1) edge[shorten >= 1pt,shorten <= 1pt] node[] {} (5)
    (6) edge[shorten >= 1pt,shorten <= 1pt] node[] {} (2)
    (3) edge[shorten >= 1pt,shorten <= 1pt] node[] {} (7)
    (8) edge[shorten >= 1pt,shorten <= 1pt] node[] {} (4) 
    (a) edge[cyan, shorten >= 1pt,shorten <= 1pt,->,bend right=20] node[] {} (b)
    (b) edge[cyan, shorten >= 1pt,shorten <= 1pt,->,bend right=20] node[] {} (c)
    (c) edge[cyan, shorten >= 1pt,shorten <= 1pt,->] node[] {} (d)
    (d) edge[cyan, shorten >= 1pt,shorten <= 1pt,->] node[] {} (e)
    (e) edge[cyan, shorten >= 1pt,shorten <= 1pt,->,bend right=20] node[] {} (f)
    (f) edge[cyan, shorten >= 1pt,shorten <= 1pt,->,bend right=20] node[] {} (g)  
    (a) edge[shorten >= 1pt,shorten <= 1pt,<-,bend left=20] node[] {} (b)
    (b) edge[shorten >= 1pt,shorten <= 1pt,<-,bend left=20] node[] {} (c)
    (e) edge[shorten >= 1pt,shorten <= 1pt,<-,bend left=20] node[] {} (f)
    (f) edge[shorten >= 1pt,shorten <= 1pt,<-,bend left=20] node[] {} (g)
    (1) edge[shorten >= 1pt,shorten <= 1pt] node[] {} (2)
    (5) edge[shorten >= 1pt,shorten <= 1pt] node[] {} (6)
    (4) edge[shorten >= 1pt,shorten <= 1pt] node[] {} (3)
    (8) edge[shorten >= 1pt,shorten <= 1pt] node[] {} (7)   
    ;
    
\end{tikzpicture}
\end{minipage}
\caption{Left: $\cF^S(\omega)$, where $\omega$ satisfies \eqref{parameter3d}. Notice that there is no path from $u$ to $u'$. Right: A partial depiction of $\cF^L(\Omega(\omega))$, where only some nodes are shown, and $\tilde u := \Psi(u)$. A path now exists from $\tilde u$ to $\tilde u'$, shown in cyan. }
\label{fig:3dpath}
\end{figure}
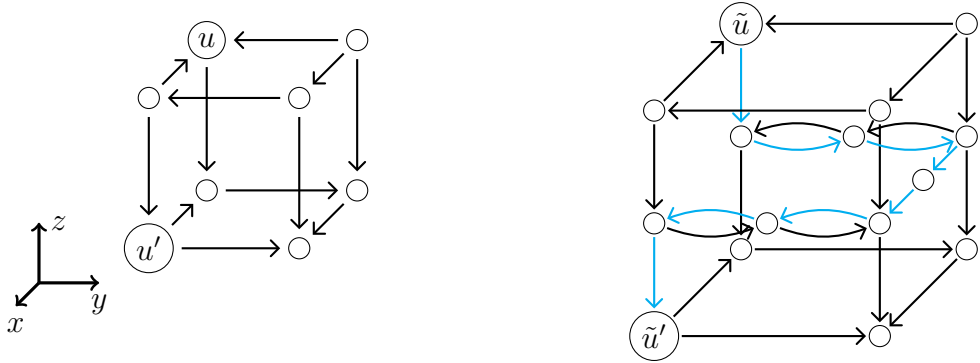

Note that the new path from $\tilde{u}$ to $\tilde{u'}$ in graph of $\cF^L$ does not have any effect on the set of Morse sets in the two systems; in both $\cF^S$ and $\cF^L$ there is unique (attracting) Morse set that consists of the state in the lower right corner of the STG.
 
To show that the existence of such a path can make a difference in the composition of the set of attractors, we will embed the structure of Figure~\ref{fig:3dpath} into a state transition graph for a four-dimensional network, such that $u$ is part of some attractor $A \in \cA^S$ and $u' \notin A$. This  will allow us to find a path escaping from the attractor in $\cF^L$. 

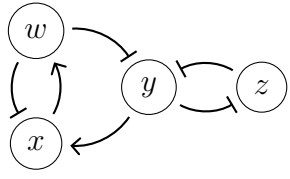
\begin{figure}[h!]
\centering
  \begin{tikzpicture}[main node/.style={circle,fill=white!20,draw,font=\sffamily\normalsize\bfseries}, scale=1]
    \node[main node] (x) at (0,0) {$x$};
    \node[main node] (y) at (1.5,.75) {$y$};
    \node[main node] (z) at (3,.75) {$z$};
    \node[main node] (w) at (0,1.5) {$w$};
        
    \path[->,>=angle 90,thick]
    (x) edge[shorten >= 3pt,shorten <= 3pt,bend right] node[] {} (w)
    (w) edge[shorten >= 3pt,shorten <= 3pt,-|,bend right] node[] {} (x)
    (w) edge[shorten >= 3pt,shorten <= 3pt,-|, bend left] node[] {} (y)
    (y) edge[shorten >= 3pt,shorten <= 3pt, bend left] node[] {} (x)
    (y) edge[shorten >= 3pt,shorten <= 3pt,-|,bend right] node[] {} (z)
    (z) edge[shorten >= 3pt,shorten <= 3pt,-|,bend right] node[] {} (y)
    ;
  \end{tikzpicture}
\caption{The network that with an S-system given by \eqref{4dattractorequations} and parameter satisfying~\eqref{4dattractorparameter} leads to the 4-dimensional attractor in Figure \ref{fig:4dattractor}.} 
\label{fig:4dattractornetwork}
\end{figure}

\vspace{12pt}

\begin{lem} \label{lem:escapefromattractor}
Consider the graphs for  for $\cF^S(\omega)$ and $\cF^L(\Omega(\omega))$ and an attractor $A \in \cA^S$. For $v_i \in A$, there can exist $v_j \not \in \cA^S$ such that a path from $\Psi(v_i)$ to $\Psi(v_j)$ exists in the graph of $\cF^L$. 
\end{lem}

\begin{proof}
Consider the network shown in Figure \ref{fig:4dattractornetwork}. We associate with this network the following equations, 
\begin{align} \label{4dattractorequations}
\dot x &= -\gamma_xx + \sigma^+_{x,y}(y)\sigma^-_{x,w}(w)\\ \nonumber
\dot y &= -\gamma_yy + \sigma^-_{y,z}(z)\sigma^-_{y,w}(w)\\ \nonumber
\dot z &= -\gamma_zz + \sigma^-_{z,y}(y)\\ \nonumber
\dot w &= -\gamma_ww + \sigma^+_{w,x}(x)
\end{align}
and consider $\omega \in \cZ^S$ satisfying
\begin{align} \label{4dattractorparameter}
l_{x,w}l_{x,y} &< \bfrac{u_{x,w}l_{x,y}}{l_{x,w}u_{x,y}} <\theta_{w,x} < u_{x,w}u_{x,y} \\ \nonumber
l_{y,w}l_{y,z} &< l_{y,w}u_{y,z} < \theta_{z,y} < u_{y,w}l_{y,z} < \theta_{x,y} < u_{y,w}u_{y,z} \\ \nonumber
l_{z,y} &< \theta_{y,z} < u_{z,y} \\ \nonumber
l_{w,x} &< \theta_{x,w} < \theta_{y,w} < u_{w,x}.
\end{align}
This parameter leads to the graph of $\cF^S$ shown in Figure~\ref{fig:4dattractor}, where the arrows between domains were assigned using the wall-labeling function $\scrL^S$ with parameter class given by ~\eqref{4dattractorparameter}. The vertices shown with square shapes are the nodes of a cyclic attractor in $\cF^S$, which can be verified to be an attractor by noting that there are no edges from any square to any circle, and that there is a path from every square node to every other square node. When considering the graph over a canonical map, $\cF^L(\Omega(\omega))$,
  the right half of the top box in Figure~\ref{fig:4dattractor} (corresponding to the lowest values of $w$) is identical to Figure~\ref{fig:3dpath}. Therefore, as in that picture, there is an escape path from $\Psi(u)$ to $\Psi(u')$ in $\cF^L$, even though there is no path from $u$ to $u'$. Since $u$ in Figure~\ref{fig:4dattractor} is in the attractor and $u'$ is not, this path demonstrates the lemma. 
\end{proof}

The above proof is based on a network with four nodes. We do not know if such an example exists for $N=3$, but we suspect it cannot without relaxing the constraints on our logic functions $M_j$. 

\begin{figure}
\centering 
\begin{tikzpicture}[main node/.style={inner sep=0pt,minimum size=8pt,circle,fill=white!20,draw}, scale=1.75]
    \draw[very thick, ->] (-1.5,-1,-1) -- (-1,-1,-1) node[anchor=north] {$y$} ;
    \draw[very thick, ->] (-1.5,-1,-1) -- (-1.5,-.5,-1) node[anchor=west]  {$z$} ;
    \draw[very thick, ->] (-1.5,-1,-1) -- (-1.5,-1,-.5) node[anchor=north]  {$x$} ;
    \draw[very thick, dashed, ->] (-1.5,-1.25,-1) -- (-1.5,-1.75,-1) node[anchor=west] {$w$} ;
    \node[main node,shape=rectangle,fill=lightgray] (1) at (0,0,0) {};
    \node[main node,shape=rectangle,fill=lightgray] (2) at (1,0,0) {};
    \node[main node,shape=rectangle,fill=lightgray] (3) at (2,0,0) {};
    \node[main node,shape=rectangle,fill=lightgray] (4) at (0,1,0) {};    
    \node[main node,shape=rectangle,fill=lightgray, inner sep=1pt] (5) at (1,1,0) {$u$};
    \node[main node,fill=lightgray] (6) at (2,1,0) {};
    \node[main node,fill=green] (7) at (0,0,1) {};
    \node[main node,fill=green, inner sep=1pt] (8) at (1,0,1) {$u'$};
    \node[main node,shape=rectangle,fill=green] (9) at (2,0,1) {};
    \node[main node,fill=green] (10) at (0,1,1) {};   
    \node[main node,fill=green] (11) at (1,1,1) {};
    \node[main node,fill=green] (12) at (2,1,1) {}; 
      
    \node[main node,shape=rectangle,fill=red] (13) at (0,-2,0) {};
    \node[main node,shape=rectangle,fill=red] (14) at (1,-2,0) {};
    \node[main node,shape=rectangle,fill=red] (15) at (2,-2,0) {};
    \node[main node,shape=rectangle,fill=red] (16) at (0,-1,0) {};    
    \node[main node,shape=rectangle,fill=red] (17) at (1,-1,0) {};
    \node[main node,fill=red] (18) at (2,-1,0) {};
    \node[main node,fill=green] (19) at (0,-2,1) {};
    \node[main node,fill=green] (20) at (1,-2,1) {};
    \node[main node,shape=rectangle,fill=green] (21) at (2,-2,1) {};
    \node[main node,fill=green] (22) at (0,-1,1) {};    
    \node[main node,fill=green] (23) at (1,-1,1) {};
    \node[main node,fill=green] (24) at (2,-1,1) {};
          
    \node[main node,shape=rectangle,fill=red] (25) at (0,-4,0) {};
    \node[main node,shape=rectangle,fill=red] (26) at (1,-4,0) {};
    \node[main node,shape=rectangle,fill=red] (27) at (2,-4,0) {};
    \node[main node,shape=rectangle,fill=red] (28) at (0,-3,0) {};    
    \node[main node,fill=red] (29) at (1,-3,0) {};
    \node[main node,fill=red] (30) at (2,-3,0) {};
    \node[main node,shape=rectangle,fill=lightgray] (31) at (0,-4,1) {};
    \node[main node,shape=rectangle,fill=lightgray] (32) at (1,-4,1) {};
    \node[main node,shape=rectangle,fill=lightgray] (33) at (2,-4,1) {};
    \node[main node,shape=rectangle,fill=lightgray] (34) at (0,-3,1) {};    
    \node[main node,fill=black!30] (35) at (1,-3,1) {};
    \node[main node,fill=black!30] (36) at (2,-3,1) {};     

  \path[->,>=angle 90,thick]  
    (4) edge[shorten >= 1pt,shorten <= 1pt] node[] {} (5)
    (6) edge[shorten >= 1pt,shorten <= 1pt] node[] {} (5)
    (10) edge[shorten >= 1pt,shorten <= 1pt] node[] {} (11)
    (12) edge[shorten >= 1pt,shorten <= 1pt] node[] {} (11)
    (1) edge[shorten >= 1pt,shorten <= 1pt] node[] {} (2)
    (2) edge[shorten >= 1pt,shorten <= 1pt] node[] {} (3)
    (7) edge[shorten >= 1pt,shorten <= 1pt] node[] {} (8)
    (8) edge[shorten >= 1pt,shorten <= 1pt] node[] {} (9)
    (1) edge[shorten >= 1pt,shorten <= 1pt] node[] {} (4)
    (5) edge[shorten >= 1pt,shorten <= 1pt] node[] {} (2)
    (6) edge[shorten >= 1pt,shorten <= 1pt] node[] {} (3)
    (7) edge[shorten >= 1pt,shorten <= 1pt] node[] {} (10)
    (11) edge[shorten >= 1pt,shorten <= 1pt] node[] {} (8)
    (12) edge[shorten >= 1pt,shorten <= 1pt] node[] {} (9)
    (7) edge[shorten >= 1pt,shorten <= 1pt] node[] {} (1)
    (10) edge[shorten >= 1pt,shorten <= 1pt] node[] {} (4)
    (11) edge[shorten >= 1pt,shorten <= 1pt] node[] {} (5)
    (6) edge[shorten >= 1pt,shorten <= 1pt] node[] {} (12)
    (8) edge[shorten >= 1pt,shorten <= 1pt] node[] {} (2)
    (3) edge[shorten >= 1pt,shorten <= 1pt] node[] {} (9)
    ;
  \path[->,>=angle 90, thick]
    (16) edge[shorten >= 1pt,shorten <= 1pt] node[] {} (17)
    (18) edge[shorten >= 1pt,shorten <= 1pt] node[] {} (17)
    (22) edge[shorten >= 1pt,shorten <= 1pt] node[] {} (23)
    (24) edge[shorten >= 1pt,shorten <= 1pt] node[] {} (23)
    (13) edge[shorten >= 1pt,shorten <= 1pt] node[] {} (14)
    (14) edge[shorten >= 1pt,shorten <= 1pt] node[] {} (15)
    (19) edge[shorten >= 1pt,shorten <= 1pt] node[] {} (20)
    (20) edge[shorten >= 1pt,shorten <= 1pt] node[] {} (21)
    (13) edge[shorten >= 1pt,shorten <= 1pt] node[] {} (16)
    (19) edge[shorten >= 1pt,shorten <= 1pt] node[] {} (22)
    (17) edge[shorten >= 1pt,shorten <= 1pt] node[] {} (14)
    (23) edge[shorten >= 1pt,shorten <= 1pt] node[] {} (20)
    (18) edge[shorten >= 1pt,shorten <= 1pt] node[] {} (15)
    (24) edge[shorten >= 1pt,shorten <= 1pt] node[] {} (21)
    (22) edge[shorten >= 1pt,shorten <= 1pt] node[] {} (16)
    (19) edge[shorten >= 1pt,shorten <= 1pt] node[] {} (13)
    (23) edge[shorten >= 1pt,shorten <= 1pt] node[] {} (17)
    (20) edge[shorten >= 1pt,shorten <= 1pt] node[] {} (14)
    (24) edge[shorten >= 1pt,shorten <= 1pt] node[] {} (18)
    (21) edge[shorten >= 1pt,shorten <= 1pt] node[] {} (15)
    ; 
  \path[->,>=angle 90, thick]
    (29) edge[shorten >= 1pt,shorten <= 1pt] node[] {} (28)
    (35) edge[shorten >= 1pt,shorten <= 1pt] node[] {} (34)
    (30) edge[shorten >= 1pt,shorten <= 1pt] node[] {} (29)
    (36) edge[shorten >= 1pt,shorten <= 1pt] node[] {} (35)
    (26) edge[shorten >= 1pt,shorten <= 1pt] node[] {} (25)
    (27) edge[shorten >= 1pt,shorten <= 1pt] node[] {} (26)
    (32) edge[shorten >= 1pt,shorten <= 1pt] node[] {} (31)
    (33) edge[shorten >= 1pt,shorten <= 1pt] node[] {} (32)
    (25) edge[shorten >= 1pt,shorten <= 1pt] node[] {} (28)
    (31) edge[shorten >= 1pt,shorten <= 1pt] node[] {} (34)
    (29) edge[shorten >= 1pt,shorten <= 1pt] node[] {} (26)
    (35) edge[shorten >= 1pt,shorten <= 1pt] node[] {} (32)
    (30) edge[shorten >= 1pt,shorten <= 1pt] node[] {} (27)
    (36) edge[shorten >= 1pt,shorten <= 1pt] node[] {} (33)
    (34) edge[shorten >= 1pt,shorten <= 1pt] node[] {} (28)
    (35) edge[shorten >= 1pt,shorten <= 1pt] node[] {} (29)
    (36) edge[shorten >= 1pt,shorten <= 1pt] node[] {} (30)
    (31) edge[shorten >= 1pt,shorten <= 1pt] node[] {} (25)
    (32) edge[shorten >= 1pt,shorten <= 1pt] node[] {} (26)
    (33) edge[shorten >= 1pt,shorten <= 1pt] node[] {} (27)
    ;
  \path[->,>=angle 90, thick]
  (18) edge[shorten >= 1pt,shorten <= 1pt, bend right=25, densely dashed] node[] {} (6)
  (19) edge[shorten >= 1pt,shorten <= 1pt, bend right=25, densely dashed] node[] {} (31)    
  ; 
\end{tikzpicture}
\caption{The $\cF^S(\omega)$ from Lemma \ref{lem:escapefromattractor}, with network shown in Figure \ref{fig:4dattractornetwork}, S-system given by~\eqref{4dattractorequations}, and parameter $\omega$ satisfying \eqref{4dattractorparameter}. Nodes in the attractor are depicted as squares, whereas nodes not in the attractor are circles. The color of each node refers to the presence and direction of outgoing edges in the $w$ direction. Green refers to an edge in the $+w$ direction, red refers to an edge in the $-w$ direction, and gray means there is no edge. Two example edges are shown as dashed arrows. The right half of the top box (corresponding to the lowest values of $w$) is identical to Figure \ref{fig:3dpath}, and so an escape path from $u$ to $u'$ exists in $\cF^L(\Omega(\omega))$ as it did previously.}
\label{fig:4dattractor}
\end{figure}
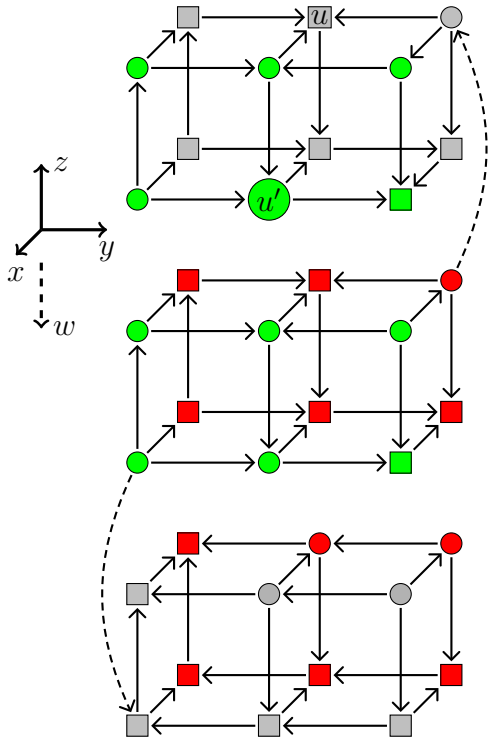

\vspace{12pt}

\begin{lem} \label{lem:phinotinjectiveingeneral}
The order-preserving map $\phi : \sMD^S \to \sMD^L$ is not necessarily injective.
\end{lem}

\begin{proof}
To show this Lemma, it is sufficient to find a network and parameters in which there exist two paths $u$ to $u'$ and $u'$ to $u$ in $\cF^L(\Omega(\omega))$, where $\Psi^{-1}(u)\in A$ and $\Psi^{-1}(u') \in B$ for some Morse sets $A,B \in \sMD^S$ with $A \neq B$.
In this way, two distinct Morse sets in the graph of $\cF^S$ will merge into one strongly connected component of $\cF^L$. To do so, we take the previous example network and embed it in a 5-dimensional network, shown in Figure~\ref{fig:5dnetwork} (left), with system of equations given by
\begin{align} \label{5dequations}
\dot x &= -\gamma_xx + \sigma^+_{x,y}(y)\sigma^-_{x,w}(w) \nonumber\\
\dot y &= -\gamma_yy + \sigma^+_{y,v}(v)\sigma^-_{y,w}(w)\sigma^-_{y,z}(z) \nonumber\\
\dot z &= -\gamma_zz + \left[ \sigma^-_{z,y}(y) + \sigma^+_{z,v}(v) \right] \\
\dot w &= -\gamma_ww + \sigma^-_{w,v}(v)\sigma^+_{w,x}(x) \nonumber\\
\dot v &= -\gamma_vv + \sigma^+_{v,x}(x)\sigma^-_{v,y}(y)\sigma^-_{v,w}(w) \nonumber
\end{align}
and $\omega \in \cZ^S$ satisfying
\begin{align}
\left\lbrace \begin{matrix}
l_{x,y}l_{x,w} \\
l_{x,y}u_{x,w} \\
u_{x,y}l_{x,w}
\end{matrix}\right\rbrace &< \theta_{w,x}< \theta_{v,x} < u_{x,y}u_{x,w} \nonumber\\
\left\lbrace \begin{matrix}
l_{y,v}l_{y,w}l_{y,z} \\
l_{y,v}l_{y,w}u_{y,z}
\end{matrix}\right\rbrace &< \theta_{z,y} < l_{y,v}u_{y,w}l_{y,z} < \theta_{v,y} < \theta_{x,y} < \left\lbrace \begin{matrix}
l_{y,v}u_{y,w}u_{y,z} \\
u_{y,v}u_{y,w}u_{y,z} \\
u_{y,v}l_{y,w}l_{y,z} \\
u_{y,v}l_{y,w}u_{y,z} \\
u_{y,v}u_{y,w}u_{y,z}
\end{matrix}\right\rbrace \nonumber\\
l_{z,y}+l_{z,v} &< \theta_{y,z} < \left\lbrace \begin{matrix}
l_{z,y}+u_{z,v} \\
u_{z,y}+l_{z,v} \\
u_{z,y}+u_{z,v}
\end{matrix}\right\rbrace \\
l_{w,v}l_{w,x} &< \left\lbrace \begin{matrix}
u_{w,v}l_{w,x} \\
l_{w,v}u_{w,x}
\end{matrix}\right\rbrace < \theta_{v,w} < \theta_{x,w} < \theta_{y,w} < u_{w,v}u_{w,x} \nonumber\\
l_{v,x}l_{v,y}l_{v,w} &< \left\lbrace \begin{matrix}
l_{v,x}l_{v,y}u_{v,w} \\
l_{v,x}u_{v,y}l_{v,w} \\
u_{v,x}l_{v,y}l_{v,w}
\end{matrix}\right\rbrace < \left\lbrace \begin{matrix}
l_{v,x}u_{v,y}u_{v,w} \\
u_{v,x}l_{v,y}u_{v,w} \\
u_{v,x}u_{v,y}l_{v,w}
\end{matrix}\right\rbrace < \theta_{w,v} < \theta_{x,v} < \theta_{y,v} < u_{v,x}u_{v,y}u_{v,w}. \nonumber
\end{align}
The full state transition graph of $\cF^S(\omega)$ is quite extensive, so we offer a 
schematic as shown in Figure~\ref{fig:schematicXCtoXC}, where each square corresponds to a 3-dimensional subset of nodes, each with coordinates $x$, $y$, and $z$.  The coordinates $v,w$ are represented in the 2D schematic. The square in cyan is shown in greater detail in Figure \ref{fig:5dsneakingpath}. Arrows between squares refer to gradient flow in the given direction, so that no paths exist in the direction opposite the arrow between any pair of nodes with the same $x,y,z$ coordinates. The dark blue rectangles labeled $\text{XC}_1$ and $\text{XC}_2$ represent cyclic Morse sets in the S-system composed of a subset of the nodes in the boxes they overlap. XC$_2$ is an unstable Morse set, because there is a path from a node in XC$_2$ to XC$_1$, which can be seen in Figure~\ref{fig:5dvrow} in the Appendix where we exhibit the full  state transition graph. The Morse graph $\sMG^S$ is shown in Figure~\ref{fig:schematicXCtoXC} (right) with only the black arrow.

In $\cF^L(\Omega(\omega))$, there is a new path in the cyan box which connects $\text{XC}_1$ to $\text{XC}_2$, as depicted in Figure~\ref{fig:5dsneakingpath} and as summarized in the Morse graph for $\cF^L(\Omega(\omega))$ shown in Figure~\ref{fig:schematicXCtoXC} (right), dashed cyan arrow. However, no such path exists in $\cF^S(\omega)$. The escape path connects a node in the attracting cycle $\text{XC}_1$ to a node from which a path exists to $\text{XC}_2$.  Therefore in the Morse graph of $\cF^L(\Omega(\omega))$, the two XC cycles merge into one stable FC cycle.
\end{proof}

\begin{figure}
\centering
\begin{tabular}{ l r }
\begin{tikzpicture}[main node/.style={circle,fill=white!20,draw,font=\sffamily\normalsize\bfseries}, scale=1]
    \node[main node] (x) at (0,0) {$x$};
    \node[main node] (y) at (3.75,1.75) {$y$};
    \node[main node] (z) at (2.2,3.2) {$z$};
    \node[main node] (w) at (0,3.5) {$w$};
    \node[main node] (v) at (1,1.75) {$v$};
        
    \path[->,>=angle 90,thick]
    (x) edge[shorten >= 3pt,shorten <= 3pt, bend left=45, red] node[] {} (w)
    (w) edge[shorten >= 3pt,shorten <= 3pt,-|,bend right=20, red] node[] {} (x)
    (w) edge[shorten >= 3pt,shorten <= 3pt,-|, out=25, in=90, red] node[] {} (y)
    (y) edge[shorten >= 3pt,shorten <= 3pt, bend left=20, red] node[] {} (x)
    (y) edge[shorten >= 3pt,shorten <= 3pt,-|,bend left=12, red] node[] {} (z)
    (z) edge[shorten >= 3pt,shorten <= 3pt,-|, red, bend left=12] node[] {} (y)
    (x) edge[shorten >= 3pt,shorten <= 3pt,->] node[] {} (v)
    (v) edge[shorten >= 3pt,shorten <= 3pt,->,bend right=15] node[] {} (y)
    (y) edge[shorten >= 3pt,shorten <= 3pt,-|,bend right=15] node[] {} (v)
    (v) edge[shorten >= 3pt,shorten <= 3pt,-|,bend right=20] node[] {} (w)
    (w) edge[shorten >= 3pt,shorten <= 3pt,-|,bend right=20] node[] {} (v)
    (v) edge[shorten >= 3pt,shorten <= 3pt,->] node[] {} (z)    
    ;
\end{tikzpicture}
&
\begin{tikzpicture}[main node/.style={circle,fill=white!20,draw,font=\sffamily\normalsize\bfseries}, scale=1]
    \node[main node] (x) at (0,0) {$x$};
    \node[main node] (y) at (3.75,1.75) {$y$};
    \node[main node] (z) at (2.2,3.2) {$z$};
    \node[main node] (w) at (0,3.5) {$w$};
    \node[main node] (v) at (1,1.75) {$v$};
        
    \path[->,>=angle 90,thick]
    (x) edge[shorten >= 3pt,shorten <= 3pt, bend left=45, red] node[] {} (w)
    (w) edge[shorten >= 3pt,shorten <= 3pt,-|,bend right=20, red] node[] {} (x)
    (w) edge[shorten >= 3pt,shorten <= 3pt,-|, out=25, in=90, red] node[] {} (y)
    (y) edge[shorten >= 3pt,shorten <= 3pt, bend left, red] node[] {} (x)
    (y) edge[shorten >= 3pt,shorten <= 3pt,-|,bend left=12, red] node[] {} (z)
    (z) edge[shorten >= 3pt,shorten <= 3pt,-|, red, bend left=12] node[] {} (y)
    (v) edge[shorten >= 3pt,shorten <= 3pt,->,bend right=20] node[] {} (x)
    (x) edge[shorten >= 3pt,shorten <= 3pt,->,bend right=20] node[] {} (v)
    (v) edge[shorten >= 3pt,shorten <= 3pt,-|,bend right=15] node[] {} (y)
    (y) edge[shorten >= 3pt,shorten <= 3pt,-|,bend right=15] node[] {} (v)
    (v) edge[shorten >= 3pt,shorten <= 3pt,-|,bend right=20] node[] {} (w)
    (w) edge[shorten >= 3pt,shorten <= 3pt,-|,bend right=20] node[] {} (v)
    ;
\end{tikzpicture}
\end{tabular}
\caption{Left: The 5-dimensional network for Lemma \ref{lem:phinotinjectiveingeneral}. Right: The 5-dimensional network for Lemma \ref{lem:phinotinjectiveonattractor}. In both cases, the embedded 4D network is shown in red.}
\label{fig:5dnetwork}
\end{figure}
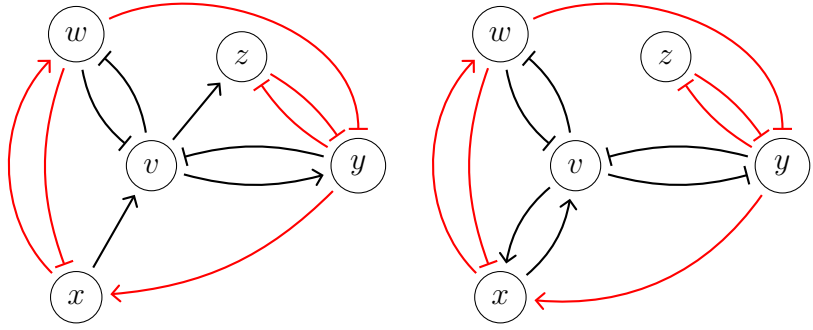

\begin{figure}
\centering
\begin{minipage}{.5\textwidth}
  \centering
\begin{tikzpicture}[main node/.style={inner sep=0,minimum size=.5cm,rectangle,draw,transform shape},scale=1.25]
\foreach \x in {1,2,3,4}
  \foreach \y in {1,2,3,4}
    \node[main node,very thick] (\x\y) at (\x,\y) {};

\draw[very thick, dashed, ->] (.25,4.75) -- ++(0,-.5) node[anchor=north]{$w$};
\draw[very thick, dashed, ->] (.25,4.75) -- ++(.5,0) node[anchor=north]{$v$};

\path[->,>=angle 90, thick]
(21) edge[shorten >= 1pt,shorten <= 1pt] node[] {} (11)
(31) edge[shorten >= 1pt,shorten <= 1pt] node[] {} (21)
(41) edge[shorten >= 1pt,shorten <= 1pt] node[] {} (31)
(22) edge[shorten >= 1pt,shorten <= 1pt] node[] {} (12)
(32) edge[shorten >= 1pt,shorten <= 1pt] node[] {} (22)
(42) edge[shorten >= 1pt,shorten <= 1pt] node[] {} (32)
(23) edge[shorten >= 1pt,shorten <= 1pt] node[] {} (13)
(33) edge[shorten >= 1pt,shorten <= 1pt] node[] {} (23)
(43) edge[shorten >= 1pt,shorten <= 1pt] node[] {} (33)

(21) edge[shorten >= 1pt,shorten <= 1pt] node[] {} (22)
(31) edge[shorten >= 1pt,shorten <= 1pt] node[] {} (32)
(41) edge[shorten >= 1pt,shorten <= 1pt] node[] {} (42)
(22) edge[shorten >= 1pt,shorten <= 1pt] node[] {} (23)
(32) edge[shorten >= 1pt,shorten <= 1pt] node[] {} (33)
(42) edge[shorten >= 1pt,shorten <= 1pt] node[] {} (43)
(23) edge[shorten >= 1pt,shorten <= 1pt] node[] {} (24)
(33) edge[shorten >= 1pt,shorten <= 1pt] node[] {} (34)
(43) edge[shorten >= 1pt,shorten <= 1pt] node[] {} (44)
;
\node[main node, cyan, very thick] () at (1,4) {};

\draw[very thick, blue] (3.125,3.875) rectangle (3.875,4.125);
\node[inner sep=1,anchor=south] (XC2) at (3.5,4.25) {$\text{XC}_2$};
\draw[ultra thick, blue] (.87,.87) rectangle (1.13,4.13);
\node[inner sep=0,anchor=east] (XC1) at (.85,2.5) {$\text{XC}_1$};
\end{tikzpicture}
\end{minipage}%
\begin{minipage}{.5\textwidth}
  \centering
\begin{tikzpicture}[main node/.style={inner sep=4pt,circle,fill=white!20,draw}]
    \node[main node] (XC1) at (0,0) {$\text{XC}_1$};
    \node[main node] (XC2) at (0,2.5) {$\text{XC}_2$};      
    \path[->,thick]
    (XC1) edge[shorten >= 3pt,shorten <= 3pt, densely dashed, cyan, very thick, bend right] node[] {} (XC2)
    (XC2) edge[shorten >= 3pt,shorten <= 3pt, very thick, bend right] node[] {} (XC1)
    ;
\end{tikzpicture}
\end{minipage}
\caption{Left: A general schematic of $\cF^S(\omega)$ from Lemma \ref{lem:phinotinjectiveingeneral}. Each square corresponds to a 3D subset of nodes, with coordinates differing in $x$, $y$, and $z$ only. The square in cyan is shown in greater detail in Figure \ref{fig:5dsneakingpath}. Arrows between squares refer to gradient flow in the given direction, i.e. no paths exist in the direction opposite the arrow. The dark blue rectangles labeled $\text{XC}_1$ and $\text{XC}_2$ represent cyclic Morse sets composed of a subset of the nodes in the boxes they overlap. In $\cF^L(\Omega(\omega))$, there is a new path in the cyan box which connects $\text{XC}_1$ to $\text{XC}_2$. However, no such path exists in $\cF^S(\omega)$. The full left column and top row are provided in Figures \ref{fig:5dwcolumn} and \ref{fig:5dvrow} (left) respectively. Right: The corresponding (partial) Morse graph. The cyan edge is added only in $\cF^L(\Omega(\omega))$, merging $\text{XC}_1$ and $\text{XC}_2$ into one strongly connected component, showing that $\phi$, in general,  is not injective between the Morse decompositions.}
\label{fig:schematicXCtoXC}
\end{figure}
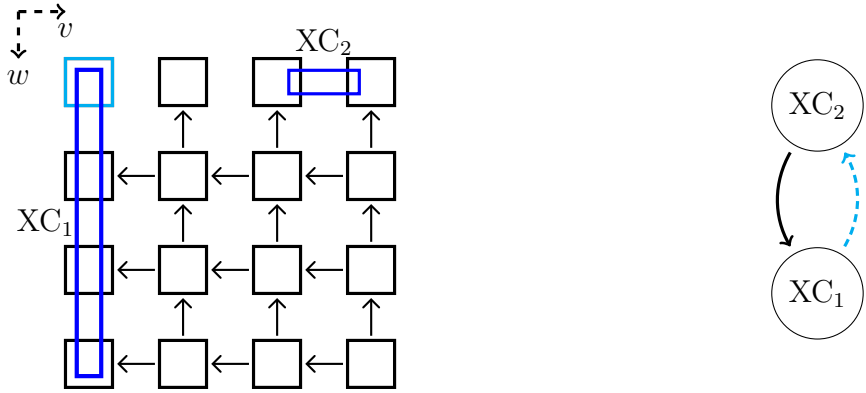

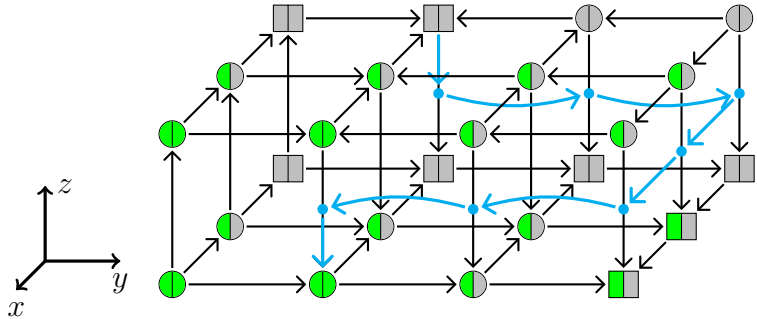
\begin{figure}
\centering
\begin{tikzpicture}[main node/.style={inner sep=0pt,minimum size=10pt,circle split,rotate=90,draw},
main node attractor/.style={inner sep=0pt,minimum size=10pt,rectangle split,rectangle split parts=2,rotate=90,draw}, scale=2]
    \draw[very thick, ->] (-2,-1,-1) -- (-1.5,-1,-1) node[anchor=north] {$y$} ;
    \draw[very thick, ->] (-2,-1,-1) -- (-2,-.5,-1) node[anchor=west]  {$z$} ;
    \draw[very thick, ->] (-2,-1,-1) -- (-2,-1,-.5) node[anchor=north]  {$x$} ;
    \node[main node attractor ,rectangle split part fill={lightgray,lightgray}] (1) at (0,0,0) {};
    \node[main node attractor,rectangle split part fill={lightgray,lightgray}] (2) at (1,0,0) {};
    \node[main node attractor,rectangle split part fill={lightgray,lightgray}] (3) at (2,0,0) {};
    \node[main node attractor,rectangle split part fill={lightgray,lightgray}] (4) at (3,0,0) {};
    \node[main node attractor,rectangle split part fill={lightgray,lightgray}] (5) at (0,1,0) {};
    \node[main node attractor,rectangle split part fill={lightgray,lightgray}] (6) at (1,1,0) {};
    \node[main node,circle split part fill={lightgray,lightgray}] (7) at (2,1,0) {};
    \node[main node,circle split part fill={lightgray,lightgray}] (8) at (3,1,0) {};
    \node[main node,circle split part fill={green,lightgray}] (9) at (0,0,1) {};
    \node[main node,circle split part fill={green,lightgray}] (10) at (1,0,1) {};
    \node[main node,circle split part fill={green,lightgray}] (11) at (2,0,1) {};
    \node[main node attractor,rectangle split part fill={green,lightgray}] (12) at (3,0,1) {};
    \node[main node,circle split part fill={green,lightgray}] (13) at (0,1,1) {};
    \node[main node,circle split part fill={green,lightgray}] (14) at (1,1,1) {};
    \node[main node,circle split part fill={green,lightgray}] (15) at (2,1,1) {};
    \node[main node,circle split part fill={green,lightgray}] (16) at (3,1,1) {};
    \node[main node,circle split part fill={green,green}] (17) at (0,0,2) {};
    \node[main node,circle split part fill={green,green}] (18) at (1,0,2) {};
    \node[main node,circle split part fill={green,lightgray}] (19) at (2,0,2) {};
    \node[main node attractor,rectangle split part fill={green,lightgray}] (20) at (3,0,2) {};
    \node[main node,circle split part fill={green,green}] (21) at (0,1,2) {};
    \node[main node,circle split part fill={green,green}] (22) at (1,1,2) {};
    \node[main node,circle split part fill={green,lightgray}] (23) at (2,1,2) {};
    \node[main node,circle split part fill={green,lightgray}] (24) at (3,1,2) {};

\path[->,>=angle 90, thick]
(5) edge[shorten >= 1pt,shorten <= 1pt] node[] {} (6)
(7) edge[shorten >= 1pt,shorten <= 1pt] node[] {} (6)
(8) edge[shorten >= 1pt,shorten <= 1pt] node[] {} (7)
(1) edge[shorten >= 1pt,shorten <= 1pt] node[] {} (2)
(2) edge[shorten >= 1pt,shorten <= 1pt] node[] {} (3)
(3) edge[shorten >= 1pt,shorten <= 1pt] node[] {} (4)
(13) edge[shorten >= 1pt,shorten <= 1pt] node[] {} (14)
(16) edge[shorten >= 1pt,shorten <= 1pt] node[] {} (15)
(9) edge[shorten >= 1pt,shorten <= 1pt] node[] {} (10)
(10) edge[shorten >= 1pt,shorten <= 1pt] node[] {} (11)
(11) edge[shorten >= 1pt,shorten <= 1pt] node[] {} (12)
(21) edge[shorten >= 1pt,shorten <= 1pt] node[] {} (22)
(23) edge[shorten >= 1pt,shorten <= 1pt] node[] {} (22)
(24) edge[shorten >= 1pt,shorten <= 1pt] node[] {} (23)
(17) edge[shorten >= 1pt,shorten <= 1pt] node[] {} (18)
(18) edge[shorten >= 1pt,shorten <= 1pt] node[] {} (19)
(19) edge[shorten >= 1pt,shorten <= 1pt] node[] {} (20)
(1) edge[shorten >= 1pt,shorten <= 1pt] node[] {} (5)
(9) edge[shorten >= 1pt,shorten <= 1pt] node[] {} (13)
(17) edge[shorten >= 1pt,shorten <= 1pt] node[] {} (21)
(6) edge[shorten >= 1pt,shorten <= 1pt] node[] {} (2)
(14) edge[shorten >= 1pt,shorten <= 1pt] node[] {} (10)
(22) edge[shorten >= 1pt,shorten <= 1pt] node[] {} (18)
(7) edge[shorten >= 1pt,shorten <= 1pt] node[] {} (3)
(15) edge[shorten >= 1pt,shorten <= 1pt] node[] {} (11)
(23) edge[shorten >= 1pt,shorten <= 1pt] node[] {} (19)
(8) edge[shorten >= 1pt,shorten <= 1pt] node[] {} (4)
(16) edge[shorten >= 1pt,shorten <= 1pt] node[] {} (12)
(24) edge[shorten >= 1pt,shorten <= 1pt] node[] {} (20)
(13) edge[shorten >= 1pt,shorten <= 1pt] node[] {} (5)
(21) edge[shorten >= 1pt,shorten <= 1pt] node[] {} (13)
(9) edge[shorten >= 1pt,shorten <= 1pt] node[] {} (1)
(17) edge[shorten >= 1pt,shorten <= 1pt] node[] {} (9)
(14) edge[shorten >= 1pt,shorten <= 1pt] node[] {} (6)
(22) edge[shorten >= 1pt,shorten <= 1pt] node[] {} (14)
(10) edge[shorten >= 1pt,shorten <= 1pt] node[] {} (2)
(18) edge[shorten >= 1pt,shorten <= 1pt] node[] {} (10)
(15) edge[shorten >= 1pt,shorten <= 1pt] node[] {} (7)
(11) edge[shorten >= 1pt,shorten <= 1pt] node[] {} (3)
(19) edge[shorten >= 1pt,shorten <= 1pt] node[] {} (11)
(8) edge[shorten >= 1pt,shorten <= 1pt] node[] {} (16)
(4) edge[shorten >= 1pt,shorten <= 1pt] node[] {} (12)
(12) edge[shorten >= 1pt,shorten <= 1pt] node[] {} (20)
;

\node[circle, fill=cyan,inner sep=0pt,minimum size=4pt] (a) at (1,.5,0) {};
\node[circle, fill=cyan,inner sep=0pt,minimum size=4pt] (b) at (2,.5,0) {};
\node[circle, fill=cyan,inner sep=0pt,minimum size=4pt] (c) at (3,.5,0) {};
\node[circle, fill=cyan,inner sep=0pt,minimum size=4pt] (d) at (3,.5,1) {};
\node[circle, fill=cyan,inner sep=0pt,minimum size=4pt] (e) at (3,.5,2) {};
\node[circle, fill=cyan,inner sep=0pt,minimum size=4pt] (f) at (2,.5,2) {};
\node[circle, fill=cyan,inner sep=0pt,minimum size=4pt] (g) at (1,.5,2) {};

\path[->,>=angle 90, thick]
(6) edge[shorten >= 1pt,shorten <= 1pt,cyan, very thick] node[] {} (a)
(a) edge[shorten >= 1pt,shorten <= 1pt,cyan, very thick, bend right=15] node[] {} (b)
(b) edge[shorten >= 1pt,shorten <= 1pt,cyan, very thick, bend right=15] node[] {} (c)
(c) edge[shorten >= 1pt,shorten <= 1pt,cyan, very thick] node[] {} (d)
(d) edge[shorten >= 1pt,shorten <= 1pt,cyan, very thick] node[] {} (e)
(e) edge[shorten >= 1pt,shorten <= 1pt,cyan, very thick, bend right=15] node[] {} (f)
(f) edge[shorten >= 1pt,shorten <= 1pt,cyan, very thick, bend right=15] node[] {} (g)
(g) edge[shorten >= 1pt,shorten <= 1pt,cyan, very thick] node[] {} (18)

(16) edge[shorten >= 1pt,shorten <= 1pt] node[] {} (24)
(23) edge[shorten >= 1pt,shorten <= 1pt] node[] {} (15)
(15) edge[shorten >= 1pt,shorten <= 1pt] node[] {} (14)
;
\end{tikzpicture}
\caption{The full set of nodes which correspond to the upper left cyan square in both Figures \ref{fig:schematicXCtoXC} and \ref{fig:schematicXCtoFP}. Nodes in $\text{XC}_1$ are denoted as squares; all other nodes are circles. The cyan path exists in $\cF^L(\Omega(\omega))$ from a node in $\text{XC}_1$ to a node not in it, but still in the 3D set of nodes, represented by the cyan box of Figures \ref{fig:schematicXCtoXC} and \ref{fig:schematicXCtoFP}. This path makes use of the same structure as the previous examples. The color of each node refers to the outgoing arrows in the $v$ and $w$ directions. The left half of each node corresponds to $w$ and the right half to $v$. Green means there is an edge from the node to the next corresponding node in the $+$ direction. Gray means no outgoing edge in the corresponding direction. There are no edges in the $-$ direction, since this graph  represents the lowest states of the $v$ and $w$ directions.
}
\label{fig:5dsneakingpath}
\end{figure}

\begin{lem} \label{lem:phinotinjectiveonattractor}
The order-preserving map $\phi : \cA^S \to \cA^L$ restricted only to attractors is, in general,  not injective.
\end{lem}

\begin{proof}
To show this we modify  the network and parameter of the previous Lemma. Consider  the network in Figure~\ref{fig:5dnetwork} right. Again, the basic structure of $\cF^S$ from Figure~\ref{fig:4dattractor} will be embedded in XC$_1$ of the first column in the schematic Figure~\ref{fig:schematicXCtoFP} (left). We will exhibit a  path from a node $u\in \cV^L$ where $\Psi^{-1}(u)\in A$ for some $A \in \cA^S$ to a node $u'\in \cV^L$ where $\Psi^{-1}(u')\in B$ for an attractor $B \neq A$.  To do so, we endow the network shown in Figure~\ref{fig:5dnetwork} right with an S-system 
\begin{align} \label{5dequations}
\dot x &= -\gamma_xx + \left[ \sigma^+_{xy}(y)+\sigma^+_{xv}(v)\right] \sigma^-_{xw}(w) \\ \nonumber
\dot y &= -\gamma_yy + \sigma^-_{yz}(z)\sigma^-_{yw}(w)\sigma^-_{yv}(v) \\ \nonumber
\dot z &= -\gamma_zz + \sigma^-_{zy}(y) \\ \nonumber
\dot w &= -\gamma_ww + \sigma^+_{wx}(x)\sigma^-_{wv}(v) \\ \nonumber
\dot v &= -\gamma_vv + \sigma^+_{vx}(x)\sigma^-_{vy}(y)\sigma^-_{vw}(w)
\end{align}
and we consider any $\omega \in \cZ^S$ satisfying
\begin{align}
(l_{x,y}+l_{x,v})l_{x,w} &<  \left\lbrace \begin{matrix}(u_{x,y}+l_{x,v})l_{x,w} \\
(l_{x,y}+u_{x,v})l_{x,w}\\
(l_{x,y}+l_{x,v})u_{x,w}\\
(l_{x,y}+u_{x,v})u_{x,w}
\end{matrix}\right\rbrace < \theta_{w,x}< \theta_{v,x} < \left\lbrace \begin{matrix}
(u_{x,y}+l_{x,v})u_{x,w} \\
(u_{x,y}+u_{x,v})l_{x,w}
\end{matrix}\right\rbrace < (u_{x,y}+u_{x,v})u_{x,w}\\ 
l_{y,z}l_{y,w}l_{y,v} &< \left\lbrace \begin{matrix}
u_{y,z}l_{y,w}l_{y,v} \\
l_{y,z}u_{y,w}l_{y,v} \\
l_{y,z}l_{y,w}u_{y,v} \\
u_{y,z}u_{y,w}l_{y,v} \\
u_{y,z}l_{y,w}u_{y,v}
\end{matrix}\right\rbrace < \theta_{z,y} < l_{y,z}u_{y,w}u_{y,v} < \theta_{v,y} < \theta_{x,y} < u_{y,z}u_{y,w}u_{y,v} \\ 
l_{z,y} &< \theta_{y,z} < u_{z,y} \\ 
l_{w,v}l_{w,x} &< \left\lbrace \begin{matrix}
u_{w,v}l_{w,x} \\
l_{w,v}u_{w,x}
\end{matrix}\right\rbrace < \theta_{v,w} < \theta_{x,w} < \theta_{y,w} < u_{w,v}u_{w,x} \\ 
l_{v,x}l_{v,y}l_{v,w} &< \left\lbrace \begin{matrix}
l_{v,x}l_{v,y}u_{v,w} \\
l_{v,x}u_{v,y}l_{v,w} \\
u_{v,x}l_{v,y}l_{v,w}
\end{matrix}\right\rbrace < \left\lbrace \begin{matrix}
l_{v,x}u_{v,y}u_{v,w} \\
u_{v,x}l_{v,y}u_{v,w} \\
u_{v,x}u_{v,y}l_{v,w}
\end{matrix}\right\rbrace < \theta_{w,v} < \theta_{z,v} < \theta_{y,v} < u_{v,x}u_{v,y}u_{v,w}.
\end{align}
A schematic  of $\cF^S(\omega)$ is shown in Figure~\ref{fig:schematicXCtoFP}. In the S-system there are two attractors, $\text{XC}_1$ which is the same as in the previous example, and a new attractor, denoted FC. Instead of a path connecting $\text{XC}_1$ to another cycle $\text{XC}_2$, as in the previous example, there is now  a path in $\cF^L(\Omega(\omega))$ from $\text{XC}_1$ to a fixed point FP.  It follows that  in $\cF^L(\Omega(\omega))$ the set $\text{XC}_1$ loses stability, and the only attractor is FP.  Therefore  $\phi$ cannot be injective even when confined to the set of attractors.
\end{proof}

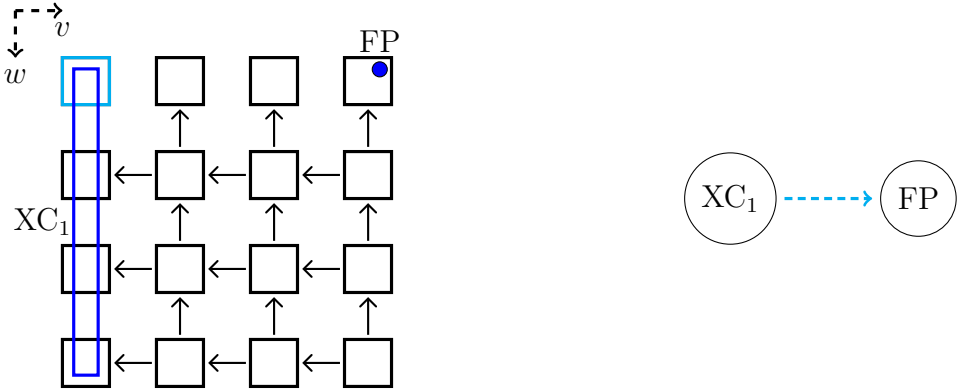
\begin{figure}[h!]
\centering
\begin{minipage}{.5\textwidth}
  \centering
\begin{tikzpicture}[main node/.style={inner sep=0,minimum size=.5cm,rectangle,draw,transform shape},scale=1.25]

\foreach \x in {1,2,3,4}
  \foreach \y in {1,2,3,4}
    \node[main node,very thick] (\x\y) at (\x,\y) {};

\draw[very thick, dashed, ->] (.25,4.75) -- ++(0,-.5) node[anchor=north]{$w$};
\draw[very thick, dashed, ->] (.25,4.75) -- ++(.5,0) node[anchor=north]{$v$};

\draw[fill=blue] (4.125,4.125) circle (.08);
\node[inner sep=2,anchor=south] (FP) at (4.125,4.25) {FP} ;

\path[->,>=angle 90, thick]
(21) edge[shorten >= 1pt,shorten <= 1pt] node[] {} (11)
(31) edge[shorten >= 1pt,shorten <= 1pt] node[] {} (21)
(41) edge[shorten >= 1pt,shorten <= 1pt] node[] {} (31)
(22) edge[shorten >= 1pt,shorten <= 1pt] node[] {} (12)
(32) edge[shorten >= 1pt,shorten <= 1pt] node[] {} (22)
(42) edge[shorten >= 1pt,shorten <= 1pt] node[] {} (32)
(23) edge[shorten >= 1pt,shorten <= 1pt] node[] {} (13)
(33) edge[shorten >= 1pt,shorten <= 1pt] node[] {} (23)
(43) edge[shorten >= 1pt,shorten <= 1pt] node[] {} (33)

(21) edge[shorten >= 1pt,shorten <= 1pt] node[] {} (22)
(31) edge[shorten >= 1pt,shorten <= 1pt] node[] {} (32)
(41) edge[shorten >= 1pt,shorten <= 1pt] node[] {} (42)
(22) edge[shorten >= 1pt,shorten <= 1pt] node[] {} (23)
(32) edge[shorten >= 1pt,shorten <= 1pt] node[] {} (33)
(42) edge[shorten >= 1pt,shorten <= 1pt] node[] {} (43)
(23) edge[shorten >= 1pt,shorten <= 1pt] node[] {} (24)
(33) edge[shorten >= 1pt,shorten <= 1pt] node[] {} (34)
(43) edge[shorten >= 1pt,shorten <= 1pt] node[] {} (44)
;
\node[main node, cyan, very thick] () at (1,4) {};
\draw[very thick, blue] (.87,.87) rectangle (1.13,4.13);
\node[inner sep=0,anchor=east] (XC1) at (.85,2.5) {$\text{XC}_1$};
\end{tikzpicture}
\end{minipage}%
\begin{minipage}{.5\textwidth}
  \centering
\begin{tikzpicture}[main node/.style={inner sep=4pt,circle,fill=white!20,draw}]
    \node[main node] (XC1) at (0,0) {$\text{XC}_1$};
    \node[main node] (FP) at (2.5,0) {$\text{FP}$};     
    \path[->,thick]
    (XC1) edge[shorten >= 3pt,shorten <= 3pt, densely dashed, cyan, very thick] node[] {} (FP);
\end{tikzpicture}
\end{minipage}
\caption{ Left: Similar to Figure \ref{fig:schematicXCtoXC}, a general schematic of $\cF^S(\omega)$ from Lemma \ref{lem:phinotinjectiveonattractor}. Each square corresponds to a 3D subset of nodes, with coordinates differing in $x$, $y$, and $z$ only. Arrows between squares refer to gradient flow in the given direction, i.e. no paths exist in the direction opposite the arrow. The square in cyan is shown in greater detail in Figure \ref{fig:5dsneakingpath}. In $\cF^L(\Omega(\omega))$, there is a new path in the cyan box which connects $\text{XC}_1$ to $\text{FP}$. The full left column and top row are provided in Figures \ref{fig:5dwcolumn} and \ref{fig:5dvrow} respectively. Right: The corresponding (partial) Morse graph. The cyan edge is added only in $\cF^L(\Omega(\omega))$, showing that $\phi$ cannot be made injective even when constrained to attractors. In effect, $\text{XC}_1$ loses stability in $\cF^L(\Omega(\omega))$.
}
\label{fig:schematicXCtoFP}
\end{figure}

\begin{rem}
The complete graphs of $\cF^S(\omega)$ and $\cF^L(\Omega(\omega))$ of the previous two Lemmas are cumbersome and do not contribute greatly to understanding on a first reading. However, the complete set of nodes corresponding to the left column and top row are included in Figures~\ref{fig:5dwcolumn} and~\ref{fig:5dvrow}. We do not include views into other 3D subgraphs, since there is gradient flow between them in $v$ or $w$ directions,  and so there are no other stable Morse sets within them.
\end{rem}

\section{Discussion}

Differential equations have been  a cornerstone of mathematical modeling of physical systems since the time of Newton. The need for predictive modeling of other continuous time processes without first principle models, like in biology, led to further expansion of these models. However, complex systems of interacting elements consisting of many equations  that are often poorly parameterized provides significant challenges to our existing paradigm of analysis of differential equations. 

Switching systems (S-systems in this paper) were proposed as a platform for modeling continuous time processes in gene regulation. The underlying assumption of these models is that regulatory genetic networks execute a Boolean function, but that this execution is embedded in a continuous flow of time  that leads to system of equations with discontinuous right hand sides. There have been many arguments about the appropriateness of such models, the technical challenges  they introduce, and, importantly, whether and how these models represent the dynamics of nearby perturbed continuous models~\cite{Ironi2011,us1}.
One of the key advantages of the S-systems is that they provide a means to combinatorialize the dynamics of the ODE system in terms of state transition graphs (STGs). These provide incomplete~\cite{glass:kaufman:73,glass:kaufman:72,edwards00} information  about the dynamics.  In our description this information is captured in a Morse graph~\cite{Chaves09,us2}. Morse graphs provide information on the number  and type of attractors present. 

In this paper we define a natural extension of state transition graphs for smooth systems that can be viewed as continuous perturbations (not necessarily small) of the  S-systems, where the perturbations are   localized in the  neighborhoods of thresholds. We call these L-systems. We  study the natural question of how the Morse graph of such a perturbation relates to a Morse graph of the S-ystem for a class of regulatory networks with no negative self-regulation. We show that there is a surjection from the set of attractors in the Morse graph of an S-system to the set of attractors of the L-system, and that this surjection is a bijection on the set of fixed point attractors. Therefore no new stable behavior can be introduced by perturbation to a smooth system. Therefore the S-system contains essential information about attractors of the smooth systems although attractors  may be lost or new unstable regions introduced by such a perturbation.

 As  important as these results are, our constructed examples of systems that show that stronger relationships that those exhibited do not exist. It is easy to construct an example  that shows that an L-system  can have more unstable Morse sets that the corresponding  S-system.
  It is a much harder task to construct  
  an L-system which combines   two attractors of the S-system into a single attractor of the L-system.
%
As a consequence of these results, one way one could hope to strengthen the results about the correspondence between the Morse graphs of the switching system and its smooth perturbations, represented by the L-system, is to refine  a definition of the STG for the L-system. Whether a refinement that would provide a closer  correspondence between Morse graphs exists is currently an open question.

\appendix

\section{Proof of Theorem~\ref{thm:Sasync}}\label{app:proofSasync}

\begin{proof}
Let $\kappa$ and $\kappa'$ be two domains with corresponding states $s_1 := g^S(\kappa)$ and $s_2 := g^S(\kappa')$, using the notation of Definition~\ref{def:asynchronousupdate}. Note that the target state of $\kappa$ is given by
\[ t_1 := D^S(s_1) = G^S \circ \Gamma^{-1}\Lambda^S(\kappa), \]
using~\eqref{eq:D^S}.

First suppose that $s_1 = s_2 = t_1$, as in part (a) of Definition~\ref{def:asynchronousupdate} of the asynchronous update rule. This is true if and only if $\Gamma^{-1}\Lambda^S(\kappa) \in \kappa$, which is equivalent to $\kappa$ being an attracting domain, which  is exactly  part (a) of Definition~\ref{def:Sdomaingraph}.

Now we consider parts (b) of the two definitions. Note that two domains $\kappa$ and $\kappa'$ are adjacent if and only if the corresponding states $s_1$ and $s_2$ are adjacent. 
So presume that $\kappa$ and $\kappa'$ share a face $\tau$, which means $s_1$ and $s_2$ are adjacent. Let $\theta_{j,i} = \pi_i(\tau)$ be the threshold at the face, and assume without loss of generality that 
$\tau$ is a right face of $\kappa$ and a left face of $\kappa'$, so that $\text{sgn}(\tau,\kappa)= -1$, $\text{sgn}(\tau,\kappa') = 1$. By the definition of $g^S$, this means that $s_{1,i} < s_{2,i}$.

Assume first that  $s_2 \in \cF^S(s_1)$ 
 so that $\scrL^S ((\tau,\kappa)) = -1$, $\scrL^S ((\tau,\kappa')) = 1$. Since $\text{sgn}(\tau,\kappa) = \scrL^S ((\tau,\kappa)) = -1$ by assumption, we have
\begin{align*}
\text{sgn}(\Lambda_i^S(\kappa)/\gamma_i^S - \theta_{j,i} ) = 1 \\
\Rightarrow \Lambda_i^S(\kappa)/\gamma_i^S > \theta_{j,i} > \pi_i(\Int \kappa) .
\end{align*}
Since $s_{1,i} = \pi_i(g^S(\kappa)) = G_i^S(x_{1,i})$ for arbitrary $x_{1,i} \in \pi_i(\Int \kappa)$, we have that 
\[  s_{1,i} = G^S_i(x_{1,i}) < G^S_i(\Lambda_i^S(\kappa)/\gamma_i^S) = t_{1,i}. \]
The statements $s_{1,i} < s_{2,i}$ and  $s_{1,i} < t_{1,i}$ then verify  Definition~\ref{def:asynchronousupdate} (b).1.

In the reverse direction, we have already proved that $s_{1,i} < s_{2,i}$ and $s_{1,i} < t_{1,i}$ imply $\scrL^S ((\tau,\kappa)) = -1$, since all statements were equivalencies. 
We show that $\scrL^S ((\tau,\kappa')) = 1$ by way of contradiction. Suppose $\scrL^S ((\tau,\kappa')) = -1$. Then 
\[ \Lambda_i^S(\kappa')/\gamma_i^S < \theta_{j,i}, \]
which implies that $x_i$ is increasing below $\theta_{j,i}$ and decreasing above $\theta_{j,i}$. This means that the  $i$-th component of the target point changes between 
$\kappa_1$ and $\kappa_2$, and these domains only differ in $i$-th coordinate. This implies that the node $i$ of the network \textbf{RN} is regulating itself; the fact that it is negative self-regulation follows from the signs of the vector field in $\kappa$ and $\kappa'$ and the fact that  $\pi_i(\Int \kappa) < \pi_i(\Int \kappa')$.
Since we assume that  \textbf{RN} has no  negative self-regulation, it must be   that $\scrL^S ((\tau,\kappa')) = 1$, as desired. 

We have shown that if  there are two adjacent domains $\kappa$ on the left and $\kappa'$ on the right and   $s_2 \in \cF^S(s_1)$ 
 according to Definition~\ref{def:Sdomaingraph} part (b), this is equivalent to condition (b).1 in Definition~\ref{def:asynchronousupdate}. We leave it to the reader to show in a similar fashion that if $s_{2,i} > s_{1,i}$ and $s_{2,i} > t_{2,i}$  (interchanging indices 1 and 2) then the  Definition~\ref{def:asynchronousupdate} (b).2 is equivalent to $s_1 \in \cF^S(s_2)$.  
 under the same domain adjacency conditions.
\end{proof}

\section{Proof of Theorem~\ref{thm:Lcompatible}}

The main result in this section is  Theorem~\ref{cornerpointproof}, from which the proof of Theorem~\ref{thm:Lcompatible} follows.

 For any $\ell$-cell  $\zeta \in\bbR^N$, $\zeta =  \prod^N_{i=1} [\varphi_i ,\varphi'_i]$ with $0 \leq \ell \leq N$ recall from Definition~\ref{def:cell} that
\[ ND(\zeta):= \{ j \;|\; \varphi \not = \varphi'\} \]
is the set of non-degenerate indices, where $|ND(\zeta)| = \ell$. Let $ND^c(\zeta) := \{1,\dots,N\} \setminus ND(\zeta)$ be the complement of $ND(\zeta)$.


\vspace{12pt}


\begin{thm} \label{straightlineproof}
Let $z^L \in Z^L$ be a regular parameter for a L-system, and let $\zeta= \prod^N_{i=1} [\varphi_i ,\varphi'_i] $ be an $\ell$-cell in $\bbR^N$ with $0 \leq \ell < N$.
For $k \in ND(\zeta)$  let $W_k:=\zeta\cap \{x_k=\varphi_k\}$ and let $W_k':=\zeta\cap \{x_k=\varphi_k'\}$. In the case where $\varphi'_k = +\infty$, choose an arbitrary point $p_k > \varphi_k$ and set  $W_k' :=\zeta\cap \{p_k=\varphi_k'\}$. 
Then for any $j \in ND^c(\zeta)$ 
    \begin{enumerate}
      \item $\dot x_j > 0$ on $W_k \cup W'_k$ implies $\dot x_j > 0$ everywhere in $\zeta$;
      \item $\dot x_j < 0$ on $W_k \cup W'_k$ implies $\dot x_j < 0$ everywhere in $\zeta$.
      \end{enumerate}
\end{thm}

\begin{proof}
Let $\zeta$ be an $\ell$-cell with $0 \leq \ell < N$ and let $k \in ND(\zeta)$. Define $W_k$ and $W'_k$ as in the theorem. Let $j\in ND^c(\zeta)$; then $\zeta \subseteq \{x_j=\varphi_j\}$ for some threshold $\varphi_j$. Define 
\[ H_j(x) := \Lambda_j(x) - \gamma_j \varphi_j\]
to be the the right-hand side of the equation for $\dot{ x}_j = H_j(x)$ on $\zeta$. Assume that $\text{sgn}(\dot x_j)$ is constant and nonzero on $W_k \cup W_k' \subset \zeta$.

Let $q\in \zeta$ be an arbitrary point. Then there exists a scalar $\alpha\geq 0$ such that $u :=q-\alpha \vec e_k\in W_k$, where $\vec e_k$ is the unit vector along  the $k$-th coordinate. Let $h:[0,1]\to \zeta$ be the line  segment along the  $k$-th coordinate direction that starts in $W_k$, passes through $Q$, and ends in $W'_k$:
\[ h(s)=u + s(\varphi_k' - \varphi_k) \vec e_k, \; 0\leq s \leq 1\]
Note that $h(0) = u \in W_k$ and $h(1)\in W'_k$ and $h(s_1) = Q$ for some $s_1 \in [0,1]$. 

We consider the two cases in which $k$ has either a regulatory effect on $j$ or not. Expressed in terms of edges in the network \textbf{RN}, this means either $(k,j)$ is an edge in \textbf{RN} ($(k,j) \in E$) or $(k,j)$  is not an edge in \textbf{RN} ($(k,j) \notin E$). First consider $(k,j) \notin E$. Then the derivative $\dot x_j  = H_j(x)$ does not depend on  $x_k$. Since the only variable that changes value along  the line segment  $h(s)$ is $x_k$,  the derivative $\dot x_j (h(s)) = H_j(h(s))$ is constant. 
 Since we know that $\dot x_j$ has the same sign everywhere on $W_k$ it must have the  same sign along $h(s)$. Since $q$ and thus $h(s)$ was arbitrary, the same is true for any point $y \in \zeta$.

Now consider the case where $(k,j) \in E$. 
We observe  that on $h(s)$, the  function $\Lambda_j$ is a linear function of $\sigma_{j,k}(h(s))$. This occurs because $\Lambda_j$ is multi-affine in $\sigma_{j,i}$ for all $i \in \mathbf{S}(j)$, the sources of $j$. But all $x_i \neq x_k$ are constant on $h(s)$,  so $\Lambda$ only changes linearly with respect to $\sigma_{j,k}(h(s))$.  We conclude  that $H_j(h(s))$ is a linear function in $\sigma_{j,k}(h(s))$.

Recall that $ l_{j,k} \leq \sigma_{j,k}(h(s)) \leq u_{j,k}$, and these bounding values are attained at the boundaries $h(s)  = \varphi_k$ and $h(s) = \varphi_k'$.  Therefore 
\[ \min\{\sigma_{j,k}(\varphi_k),\sigma_{j,k}(\varphi_k')\} \leq \sigma_{j,k}(h(s)) \leq \max\{\sigma_{j,k}(\varphi_k),\sigma_{j,k}(\varphi_k')\}. \]
From this inequality and the linearity of $H_j$ in $\sigma_{j,k}(h(s))$, we  conclude
\[ \min \{ H_j(h(0)), H_j(h(1))\}  \leq  H_j(h(s)) \leq \max  \{ H_j(h(0)), H_j(h(1))\} \]
for all $s\in [0,1]$.

When $\dot x_j > 0$ on $W_k \cup W'_k$, then $H_j (h(0))>0$ and $H_j (h(1))>0$, which implies $H_j(h(s)) > 0$ for $0 \leq s \leq 1$. Likewise, when $\dot x_j < 0$ on $W_k \cup W'_k$, then $H_j (h(0))<0$ and $H_j (h(1))<0$, which implies $H_j(h(s)) < 0$ for $0 \leq s \leq 1$. Since the selection of $q$ and hence $h(s)$  was arbitrary, we have proven that the sign of $\dot x_j$ on $\zeta$ is determined by its sign on $W_k \cup W'_k$.
\end{proof}

\vspace{12pt}

\begin{thm} \label{cornerpointproof} 
Let $z^L$ be a regular parameter for an L-system, and let $ \zeta := \prod^N_{i=1} [\varphi_i ,\varphi'_i].$ be an $\ell$-cell in $\bbR^N$ with $0 \leq \ell < N$.
Then for all $j \in ND^c(\zeta)$ we have 
    \begin{enumerate}
      \item $\sgn(\cC(\zeta),j) = +1$ implies $\dot x_j > 0$ everywhere in $\zeta$;
      \item $\sgn(\cC(\zeta),j) = -1$ implies $\dot x_j < 0$ everywhere in $\zeta$.
    \end{enumerate}
\end{thm}

\begin{proof}
Suppose $\ell=0$. Then $\cC(\zeta)=\zeta$ and the proof is immediate. This is the base case for an inductive proof.
Let $\zeta$ be an $\ell$-cell with $1 \leq \ell < N$ and 
assume that for all $(\ell-1)$-cells the theorem holds. Let $j \in ND^c(\zeta)$ be a degenerate index.

First assume that $\sgn(\cC(\zeta),j)=+1$. Pick any $k \in ND(\zeta)$, let $[\varphi_k ,\varphi'_k]$ be the corresponding  non-degenerate interval and let $W_k:=\zeta\cap \{x_k=\varphi_k\}$, $W'_k:=\zeta\cap \{x_k=\varphi'_k\}$. 
Notice that by Definition \ref{def:cell}, $W_k$ and $W'_k$ are $(\ell-1)$-cells. Furthermore $\cC(W_k)\subseteq\cC(\zeta)$, and so by Definition \ref{def:sgn}, $\sgn(\cC(W_k),j)=+1$. By our inductive hypothesis, $\dot x_j > 0$ everywhere in $W_k$. A similar argument shows that $\dot x_j > 0$ everywhere in $W'_k$ as well. Then by Theorem \ref{straightlineproof}, $\dot x_j >0$ everywhere in $\zeta$.

A similar argument is used when $\sgn(\cC(\zeta),j)=-1$.
This finishes the induction and hence the proof.
\end{proof}


\section{Lemmas for Section~\ref{sec:Morsegraphs}}

In the following, we assume that constructions in the S- and L-systems come from equivalence class parameters $\omega^S$ and $\omega^L := \Omega(\omega^S)$.

\vspace{12pt}

\begin{defn} \label{def:domainnodecorrespondence}
We define the bijection 
\[ \beta = (g^L)^{-1} \circ \Psi \circ g^S : \cK^S \to \cK^L_{N} \]
and the order-preserving functions
\begin{align*}
  \beta^-(\theta_{j,i}) &= \vartheta_{j,i}^- \\
  \beta^+(\theta_{j,i}) &= \vartheta_{j,i}^+. 
\end{align*}
\end{defn}

When $\kappa$, $\kappa' \in \cK^S$ are adjacent with shared face $\tau \subset \{ x_i = \theta_{j,i}\}$, consider the domain $\eta \in \cK^L_{N-1}$
\[ \eta =  \prod\limits_{k = 1}^{i-1} \beta_k(\kappa) \times [\beta^-(\theta_{j,i}),\beta^+(\theta_{j,i})] \times  \prod\limits_{k = i+1}^{N} \beta_k(\kappa). \]
It is easy to see that
\begin{enumerate}
  \item $\eta$ shares a face with $\beta(\kappa)$, $ \tau^- := \eta \cap \beta(\kappa) \subset \{x_i = \vartheta^-_{j,i}\}$
  \item $\eta$ shares a face with $\beta(\kappa')$, $ \tau^+ := \eta \cap \beta(\kappa') \subset \{x_i = \vartheta^+_{j,i}\}$.
\end{enumerate}
In other words, $\eta$ is the unique domain that lies between $\beta(\kappa)$ and $\beta(\kappa')$, and this domain is in the subset $\cK^L_{N-1} \subset \cK^L$.

The next Lemma is the key result from which many results about the correspondence between the Morse graphs follow.

\vspace{12pt}

\begin{lem}\label{lem:equivpath}

  Consider   $\cF^S(\omega^S)$ and of $\cF^L(\Omega(\omega^S))$ and two adjacent  domains $\kappa, \kappa' \in \cK^S$ with shared face $\tau \subset \{ x_i = \theta_{j,i}\}$.  Let $\zeta := \beta(\kappa), \zeta' := \beta(\kappa')$, $\zeta, \zeta' \in \cK^L_N$ be the corresponding domains in $\cK^L_N$ and let 
  $\eta \in \cK^L_{N-1}$ be the unique domain lying between $\zeta$ and $\zeta'$. 
  Let  $v:= g^S(\kappa), v':= g^S(\kappa')$, $v,v' \in \cV^S$, let  $w:= g^L(\zeta), w':=g^L (\zeta'), w,w' \in \cV^{SL}$, and  $u = g^L(\eta)$.
  Then $v \to v' \in \cF^S(\omega)$ if and only if $w \to u \to w' \in \cF^L(\Omega(\omega))$.
   \end{lem}

  \begin{proof}
    We consider  the case when $(\tau,\kappa)$ is the right wall of $\kappa$ ($\text{sgn}((\tau,\kappa))=-1$) and $(\tau,\kappa')$ is the left wall of $\kappa'$ ($\text{sgn}((\tau,\kappa))=+1$). A similar argument holds in the other  case. Then 
    $v' \in \cF^S(v) $
     if and only if
     $\scrL^S((\tau,\kappa)) = -1$  and $\scrL^S((\tau,\kappa')) = +1 $, which by Definition~\ref{def:Swall_labeling},   implies that 
     \[ \Lambda^S_i(\kappa)/\gamma^S> \theta_{j,i}, \quad \Lambda^S_i(\kappa')/\gamma^S > \theta_{j,i} .\]
     By Definition~\ref{def:canonicalperturbation} of the correspondence $\Omega$ we know that $D^S\circ g^S(\kappa) = D^L_N \circ g^L(\zeta)$, so if $\Lambda^S_i(\kappa)/\gamma^S> \theta_{j,i}$, then 
    \[ \Lambda^L_i(\zeta)/\gamma^L_i > \beta^+(\theta_{j,i}) \quad \Lambda^L_i(\zeta')/\gamma^L_i > \beta^+(\theta_{j,i}).
    \] 

  Then from  
      Definition~\ref{def:sgn} we get $ \sgn(\cC(\tau^- ),i)= \sgn(\cC(\tau^+) ,i) = +1 $. We note that $\tau^-$ is a right face of $\zeta$ and a left face of $\eta$, while $\tau^+$ is a right face of $\eta$ and a left face of $\zeta'$, by the assumptions of the Lemma. 
     With this information  we can compute 
     \begin{align*}
      \scrL^L((\tau^-,\zeta)) = -1 \cdot +1 = -1, \quad &  \scrL^L((\tau^-,\eta)) = +1 \cdot +1 =+1\\
     \scrL^L((\tau^+,\eta)) = -1 \cdot +1 = -1, \quad & \scrL^L((\tau^+,\zeta'))= +1 \cdot +1 =+1 .
     \end{align*}
     Finally, by Definition~\ref{def:sgn} this is equivalent with the existence of a path $w \to u \to w'$ in $\cV^L$. Since all the previous statements are equivalencies this finishes the proof.
\end{proof}

\vspace{12pt}

\begin{cor}\label{cor:2npath}
Consider   $\cF^S(\omega^S)$ and of $\cF^L(\Omega(\omega^S))$. For any two $v,v' \in \cV^S$ let $w:=\Psi(v), w':=\Psi(v')$.
 Then $v' \in (\cF^S)^k(v)$ for some integer $k$ (which means there is a path of length $k$ in the graph $(\cV^S,\cE^S)$), if and only if $w' \in (\cF^L)^{2k}(w)$, where every domain $\kappa_i = (g^L)^{-1}(w_i)$ belongs to $ \cK^L_N \cup \cK^L_{N-1}$ for all nodes  $w_i$ in the path.
  
 \end{cor}


\vspace{12pt}

\begin{lem} \label{lem:elltoellplus1}
 Consider a  regulatory network \textbf{RN}, an L-system with regular parameter $z^L \in Z^L$, the set of domains $\cK^L$ and nearest neighbor multi-valued map $\cF^L$. 
 Let $\kappa \in \cK^L\setminus \cK^L_N$ be a domain and let  $u = g^L(\kappa)$ have $k$ non-integer components. Then there is a state $v \in \cV^L$ with $k- 1$ non-integer components, such that 
  \[ v \in \cF^L(u).\]

\end{lem}
  
\begin{proof}
  Let  $\kappa \in \cK^L\setminus \cK^L_N$ with $u = g^L(\kappa)$. Then there is an index $i$ such that $\pi_i(\kappa) = [\vartheta_{j,i}^-, \vartheta_{j,i}^+]$. Let  $\tau^-$ and $\tau^+$ be the left and right faces of $\kappa$ with projection index $i$ and so  $\tau^- \subseteq \{x_i =\vartheta_{j,i}^- \}$ and $\tau^+ \subseteq \{x_i = \vartheta_{j,i}^+ \}$.
  Note that there are two domains $\eta^-, \eta^+$ that are immediate neighbors of $\kappa$ along the $i$-th coordinate which satisfy 
  \[ \pi_i(\eta^-) = [\vartheta_{j-1,i}^+, \vartheta_{j,i}^-] \qquad \pi_i(\eta^+) = [\vartheta_{j,i}^+, \vartheta_{j+1,i}^-]. \]
  It follows that the states 
  \[v^- := g^L(\eta^-) \mbox{ and } v^+ := g^L(\eta^+).\]
  have one  more integer value than $u$.

  Let $\cC(\tau^-)$ be a collection of corner points of $\tau^-$ and $\cC(\tau^+)$ be a collection of corner points of $\tau^+$. Note that there is a bijection $\alpha$ between these two sets and the corresponding corner points that only differ in the $i$-th values where $\vartheta_{j,i}^-$ is replaced by $\vartheta_{j,i}^+$. 
  Take $q \in \cC(\tau^-)$ and assume first that $\sgn(q,i) = +1$. This implies that 
  $\Lambda_i^L(q)/\gamma_i^L > \vartheta_{j,i}^-$. But since at any regular parameter $z^L$ the value 
   \[ \Lambda_i^L(q)/\gamma_i^L \not \in  [\vartheta_{j,i}^-, \vartheta_{j,i}^+] \]
  for any $j$, we conclude that also  $\Lambda_i^L(q)/\gamma_i^L > \vartheta_{j,i}^+$.
  This in turn implies that at the  corner point $\alpha(q) \in \cC(\tau^+)$ we have  $\sgn(\alpha(q),i) = +1$.  We have shown that 
  \[ \sgn(q,i) = +1 \quad  \mbox{if and only if} \quad \sgn(\alpha(q),i) = +1 .\]
   A similar argument shows that $\sgn(q,i) = -1$ if and only if $\sgn(\alpha(q),i) = -1$
  as well. 
  Let $u = g^L(\kappa)$. From the definition of the map $\cF^L$ we have 
  \begin{itemize}
  \item  if there exists at least one corner point $q \in  \cC(\tau^-)$ with $\sgn(q,i) = -1$, then \[v^- \in \cF^L(u);\]
  \item  if there exists at least one corner point $q \in  \cC(\tau^+)$ with $\sgn(q,i) = +1$, then \[v^+ \in \cF^L(u).\]
  \end{itemize}
  Since the $\sgn$ of a single corner point can never be zero by the regularity of $z^L$, we conclude that either one
   or both of the cases hold.
\end{proof}

\section{State transition graphs for the 5D examples}

We now present  full information about the paths in the proofs of Lemmas~\ref{lem:phinotinjectiveingeneral} and~\ref{lem:phinotinjectiveonattractor}. The first column (Figure~\ref{fig:5dwcolumn}) and the first rows (Figure~\ref{fig:5dvrow})  of  the schematics in Figures~\ref{fig:schematicXCtoXC} and~\ref{fig:schematicXCtoFP} are shown below.  The rows have been rotated into columns to make them more legible, and are arranged for side-by-side comparison.

\include{largetikzgraphs3}

{\bf Acknowledgements:} T. G. was partially supported by  NSF grants DMS-1226213, DMS-1361240, USDA 2015-51106-23970,  DARPA  grants D12AP200025 and FA8750-17-C-0054,  and NIH grants 1R01AG040020-01 and 1R01GM126555-01. B.C. was partially supported by grants USDA 2015-51106-23970, DARPA  grants D12AP200025 and FA8750-17-C-0054 and NIH 1R01GM126555-01. P. C-K. was supported by USP and INBRE student research grants at Montana State University. Research reported in this publication was supported by the National Institute of General Medical Sciences of the National Institutes of Health under Award Number P20GM103474. The content is solely the responsibility of the authors and does not necessarily represent the official views of the National Institutes of Health.

\bibliography{bigdata}{}
\bibliographystyle{plain}

\end{document}